\documentclass[11pt]{amsart}

\usepackage{amssymb,amsmath,amscd,stmaryrd}
\usepackage{hyperref,MnSymbol,enumitem,graphicx}

\newcommand{\A}{\mathbb{A}}
\newcommand{\Q}{\mathbb{Q}}
\newcommand{\R}{\mathbb{R}}
\newcommand{\Z}{\mathbb{Z}}
\newcommand{\F}{\mathbb{F}}
\newcommand{\p}{\mathfrak p}
\newcommand{\q}{\mathfrak q}
\renewcommand{\P}{\mathfrak P}
\newcommand{\calA}{\mathcal A}
\newcommand{\calB}{\mathcal B}
\newcommand{\calC}{\mathcal C}
\newcommand{\calE}{\mathcal E}
\newcommand{\calF}{\mathcal F}
\newcommand{\calG}{\mathcal G}
\newcommand{\calH}{\mathcal H}
\newcommand{\calI}{\mathcal I}
\newcommand{\calO}{\mathcal O}
\newcommand{\calR}{\mathcal R}
\newcommand{\calX}{\mathcal X}
\newcommand{\calW}{\mathcal W}

\renewcommand{\to}{\rightarrow}
\newcommand{\into}{\hookrightarrow}
\newcommand{\onto}{\twoheadrightarrow}
\newcommand{\disc}{\operatorname{disc}}

\renewcommand{\Im}{\operatorname{Im}} 
\newcommand{\Gal}{\operatorname{Gal}}
\newcommand{\Jac}{\operatorname{Jac}}
\newcommand{\Res}{\operatorname{Res}} 
\newcommand{\Spec}{\operatorname{Spec}} 
\newcommand{\Ann}{\operatorname{Ann}}
\newcommand{\ord}{\operatorname{ord}}

\newtheorem{thm}{Theorem}[section]
\newtheorem{lem}[thm]{Lemma}
\newtheorem{prop}[thm]{Proposition}
\newtheorem{cor}[thm]{Corollary}
\theoremstyle{definition}
\newtheorem{rem}[thm]{Remark}
\newtheorem{prob}[thm]{Problem}

\newtheorem{notation}[thm]{Notation}
\numberwithin{equation}{section}

\title[Galois groups of dynatomic polynomials]{Galois groups in a family of dynatomic polynomials}

\author{David Krumm}
\address{Department of Mathematics and Statistics\\
Colby College}
\email{david.krumm@colby.edu}
\urladdr{http://colby.edu/\textasciitilde dkrumm/}

\begin{document}
\maketitle
\begin{abstract}
For every nonconstant polynomial $f\in\Q[x]$, let $\Phi_{4,f}$ denote the fourth dynatomic polynomial of $f$. We determine here the structure of the Galois group and the degrees of the irreducible factors of $\Phi_{4,f}$ for every quadratic polynomial $f$. As an application we prove new results related to a uniform boundedness conjecture of Morton and Silverman. In particular we show that if $f$ is a quadratic polynomial, then, for more than $39\%$ of all primes $p$, $f$ does not have a point of period four in $\Q_p$.
\end{abstract}

\section{Introduction}

Let $f\in\Q[x]$ be a nonconstant polynomial. For every positive integer $n$, let $f^n$ denote the $n$-fold composition of $f$ with itself. An algebraic number $\alpha$ is called \textit{periodic} under iteration of $f$ if there exists $n\ge 1$ such that $f^n(\alpha)=\alpha$; in that case the least such $n$ is called the \textit{period} of $\alpha$.

A fundamental conjecture of Morton and Silverman \cite{morton-silverman} would imply that as $f$ varies over all polynomials of fixed degree $d>1$, the possible periods of rational numbers under iteration of $f$ remain bounded. In \cite{poonen_preperiodic} Poonen studied the case $d=2$ and conjectured that no quadratic polynomial over $\Q$ has a rational point of period greater than 3.

A useful construction for studying these conjectures is that of a dynatomic polynomial. For every nonconstant polynomial $f\in\Q[x]$ and every positive integer $n$, the $n^{\text{th}}$ \textit{dynatomic polynomial} of $f$ is defined by the equation
\begin{equation}\label{dynatomic_poly_defn}
\Phi_{n,f}(x)=\prod_{d|n}\left(f^d(x)-x\right)^{\mu(n/d)},
\end{equation}
where $\mu$ is the M\"{o}bius function. The key property that motivates this definition is that the roots of $\Phi_{n,f}$ are precisely the algebraic numbers having period $n$ under iteration of $f$, except in rare cases when some roots may have period smaller than $n$; see \cite[Thm. 2.4]{morton-patel} for further details. By studying algebraic properties of the dynatomic polynomials of $f$ one may thus hope to gain information about the dynamical properties of $f$ as a map $\bar\Q\to\bar\Q$. With this in mind, let us now focus on the case of quadratic polynomials and discuss the following problem, which we consider to be especially important for understanding the dynamics of such maps.

\begin{prob}\label{dynatomic_problems}
Given $n\ge 1$, determine all possible groups that can arise as the Galois group of $\Phi_{n,f}$ for some quadratic polynomial $f\in\Q[x]$. Furthermore, determine all possible factorization types\footnote{By the \textit{factorization type} of a polynomial $F\in\Q[x]$ we mean the multiset of degrees of irreducible factors of $F$.} of $\Phi_{n,f}$ that can arise as $f$ varies over all quadratic polynomials.
\end{prob}

This problem is easily solved for $n=1$ and 2. The case $n=3$ is substantially harder, but was solved by Morton \cite{morton_period3}. The purpose of this article is to treat the case $n=4$. At present there appears to be no published work concerning the structure of the Galois groups of the polynomials $\Phi_{n,f}$ in this case. Regarding factorization types we are only aware of two results in the literature: Morton \cite[Thm. 4]{morton_period4} showed that $\Phi_{4,f}$ can never have a factor of degree 1, and Panraksa \cite[Thm. 2.3.5]{panraksa_thesis} showed that $\Phi_{4,f}$ cannot have four or more irreducible quadratic factors.

In order to state our results for $n=4$ we introduce some notation. For every quadratic polynomial $f\in\Q[x]$ there exist a unique rational number $c$ and a unique linear polynomial $\mathit{l}\in\Q[x]$ such that
\[\mathit{l}\circ f\circ\mathit{l}^{-1}=\phi_c(x):=x^2+c.\]

The polynomials $f$ and $\phi_c$ share all the properties we are concerned with in this article; in particular, their dynatomic polynomials factor in the same way and have the same Galois group (see \cite[\S 2.2]{krumm_lgp}). In stating our results we may therefore restrict attention to the family of polynomials $\{\phi_c(x):c\in\Q\}$. To ease notation we will write $\Phi_{4,c}$ instead of $\Phi_{4,\phi_c}$.

Let $\lambda, \eta, \rho\in\Q(z)$ be the rational functions defined by
\begin{equation}\label{rational_fncn_def}
\lambda(z)=\frac{z^2 + 2z - 4}{8z},\;\;\;\eta(z)=\frac{4-3z-z^3}{4z},\;\;\;\rho(z)=\frac{1+4z^3-z^6}{4z^2(z^2-1)}.
\end{equation}

The set of all numbers of the form $\lambda(r)$ with $r\in\Q\setminus\{0\}$ will be denoted by $\Im \lambda$. Similarly, we define sets $\Im \eta$ and $\Im \rho$. We can now state our main result.

\begin{thm}\label{phi4_factorization_galois_intro_thm}
Let $c\in\Q\setminus \left\{ -5, -\frac{5}{2}, -\frac{155}{72}, -2, -\frac{5}{4}, 0, \frac{19}{16}\right\}$. Let $G_c$ and $\calF_c$, respectively, denote the Galois group and factorization type of $\Phi_{4,c}$. Referring to the groups $\calW,\calG,\calH$, and $\calI$ defined in Appendix \ref{galois_data_appendix}, we have the following:
\begin{enumerate}
\item If $c\notin\Im(\lambda)$ and $c\notin\Im(\eta)$, then $G_{c}\cong\calW$ and $\calF_c=\{12\}$.
\item If $c\in\Im(\lambda)$, then $G_c\cong \calG$ and $\calF_c=\{12\}$.
\item Suppose that $c\in\Im(\eta)$. Then the following hold:
\begin{enumerate}
\item If $c\not\in\Im(\rho)$, then $G_c\cong\calH$ and $\calF_c=\{8,4\}$.
\item If $c\in\Im(\rho)$, then $G_c\cong\calI$ and $\calF_c=\{8,2,2\}$.
\end{enumerate}
\end{enumerate}
\end{thm}

For the numbers $c$ excluded in the above theorem, the factorization types $\calF_c$ and isomorphism types of $G_c$ are given in Table \ref{excluded_gc_table}.

In addition to solving Problem \ref{dynatomic_problems} for $n=4$, Theorem \ref{phi4_factorization_galois_intro_thm} provides a way to improve on earlier results. As mentioned above, Morton showed that there is no $c\in\Q$ such that $\phi_c$ has a rational point of period 4. We can obtain this theorem as a consequence of Theorem \ref{phi4_factorization_galois_intro_thm}, as it is simply the observation that $1\notin\calF_c$. Morton's theorem was strengthened by the author in \cite{krumm_lgp}, where it is shown that for every quadratic polynomial $f$ there exist infinitely many primes $p$ such that $f$ does not have a point of period four in $\Q_p$. By using the information about Galois groups afforded by Theorem \ref{phi4_factorization_galois_intro_thm}, we can now prove the following stronger statement.

\begin{thm}\label{period_percentage_thm}
Let $f\in\Q[x]$ be a quadratic polynomial. Then, for more than $39\%$ of all primes $p$ (in the sense of Dirichlet density), $f$ does not have a point of period four in $\Q_p$.
\end{thm}

This article is organized as follows. In \S\ref{dynatomic_section} we review the necessary background material on dynatomic polynomials. In \S\ref{hit_section} we develop the main tools used in the proof of Theorem \ref{phi4_factorization_galois_intro_thm}. In \S\ref{phi4_factorization_galois_section} we prove this theorem and then deduce Theorem \ref{period_percentage_thm}. Finally, in \S\ref{rational_points_section} we determine the rational points on a number of hyperelliptic curves that arise at various stages of the proof.

\section{Dynatomic polynomials}\label{dynatomic_section}

Let $K$ be a field of characteristic $0$ and let $f\in K[x]$ be a nonconstant polynomial. For every positive integer $n$, the $n^{\text{th}}$ dynatomic polynomial of $f$ is defined by the formula \eqref{dynatomic_poly_defn}.

Two special cases of this definition are especially relevant here. First, the case where $K=\Q$ and $f$ is a polynomial of the form $\phi_c(x)=x^2+c$. In this case the polynomial $\Phi_{n,f}\in\Q[x]$ will be denoted $\Phi_{n,c}(x)$. Second, the case where $K$ is the rational function field $\Q(t)$ and $f(x)=x^2+t$. In this case the polynomial $\Phi_{n,f}\in\Q(t)[x]$ will be denoted $\Phi_n(t,x)$. By a basic property of dynatomic polynomials (see \cite[Thm. 3.1]{morton-patel}), for every $c\in\Q$ the specialization $\Phi_n(c,x)$ is equal to the $n^{\text{th}}$ dynatomic polynomial of $\phi_c$. Hence $\Phi_n(c,x)=\Phi_{n,c}(x)$.

The Galois group and factorization type of the polynomial $\Phi_n(t,x)$ are well understood. Indeed, Bousch \cite[Chap. 3]{bousch} showed that $\Phi_n(t,x)$ is irreducible in $\mathbb C(t)[x]$ and determined its Galois group over $\mathbb C(t)$. It follows from Bousch's results that $\Phi_n(t,x)$ is irreducible in $\Q(t)[x]$ and its Galois group over $\Q(t)$ is isomorphic to the wreath product of $\Z/n\Z$ with the symmetric group $S_r$, where $r=(\deg\Phi_n)/n$. Thus, the problem of classifying the Galois groups and factorization types of the polynomials $\Phi_{n,c}$ for $c\in\Q$ can be regarded as a special case of the following more general problem.
\begin{prob} Given a polynomial $P\in\Q[t,x]$ with known Galois group and factorization type, determine the Galois group and factorization type of every specialization $P(c,x)$ with $c\in\Q$.
\end{prob}
In the next section we discuss a technique that can be used to solve this problem under certain assumptions on the polynomial $P$.

\section{Explicit Hilbert Irreducibility}\label{hit_section}

Let $P\in\Q[t,x]$ be a polynomial of degree $n\ge 1$ in the variable $x$. Regarding $P$ as an element of the ring $\Q(t)[x]$, let $\ell(t)$ be its leading coefficient, $\Delta(t)$ its discriminant, and $\calF(P)$ its factorization type, i.e., the multiset of degrees (in $x$) of the irreducible factors of $P$. Let $N$ be a splitting field for $P$ and let $G=\Gal(N/\Q(t))$ be the Galois group of $P$.

For every rational number $c$ we let $P_c=P(c,x)$ and we denote by $G_c$ and $\calF(P_c)$, respectively, the Galois group and factorization type of $P_c$. By Hilbert's Irreducibility Theorem \cite[Prop. 3.3.5]{serre_topics}, for every number $c$ outside of a thin subset of $\Q$ we have $G_c\cong G$ and $\calF(P_c)=\calF(P)$. We define the \textit{exceptional set} of $P$, denoted $\calE(P)$, to be the set of rational numbers $c$ for which either one of these conditions fails:
\[\calE(P)=\{c\in\Q : G_c\not\cong G\;\;\text{or}\;\;\calF(P_c)\ne\calF(P)\}.\]

The purpose of this section is to develop a technique that can be used to explicitly describe the set $\calE(P)$ and to determine $G_c$ and $\calF(P_c)$ for every $c\in\calE(P)$, especially in the case where $P$ is irreducible. For a more general treatment of this subject we refer the reader to the article \cite{krumm-sutherland}.

Let $A\subset\Q(t)$ be the ring $A=\Q[t][\ell(t)^{-1}]$ and let $\calO_N$ be the integral closure of $A$ in $N$. If $\P$ is a maximal ideal of $\calO_N$ (henceforth referred to as a \textit{prime} of $N$), we denote by $G_{\P}$ the decomposition group of $\P$ over $\Q(t)$.

For all $c\in\Q$ satisfying $\ell(c)\ne 0$, let $\p_c$ be the kernel of the evaluation homomorphism $A\to\Q$ given by $a(t)\mapsto a(c)$.

Throughout this paper we will use the following notation to distinguish between an isomorphism of groups and an isomorphism of groups with an action.

\begin{notation}[Isomorphism of group actions]
Suppose that the groups $G$ and $H$ act on sets $X$ and $Y$, respectively. We write $G\equiv H$ if there exist an isomorphism $\varphi:G\to H$ and a bijection $\sigma:X\to Y$ such that $\sigma(g\cdot x)=\varphi(g)\cdot\sigma(x)$ for all $g\in G$ and all $x\in X$.
\end{notation}

\begin{lem}\label{HIT_lem}
Suppose that $c\in\Q$ satisfies $\Delta(c)\cdot\ell(c)\ne 0$. Then the following hold:
\begin{enumerate}
\item\label{decomposition_group_iso} Let $\P$ be a prime of $N$ dividing $\p_c$. Then $G_{\P}\equiv G_c$, where $G_{\P}$ acts on the roots of $P$ and $G_c$ acts on the roots of $P_c$.
\item\label{factorization_vs_Galois_lem} If $G_c\cong G$, then $\calF(P_c)=\calF(P)$.
\end{enumerate}
\end{lem}
\begin{proof}
Statement (1) is well known, at least in the case where $P$ is monic and  irreducible; a proof can be found in \cite[Chap. VII, Thm. 2.9]{lang_algebra}. We include a sketch of the general case here.

For every element $a\in\calO_N$ let $\bar a$ denote the reduction of $a$ modulo $\P$. The extension of residue fields $\kappa(\P)/\Q$ is Galois, and there is a surjective homomorphism $G_{\P}\to\Gal(\kappa(\P)/\Q)$ given by $\sigma\mapsto\bar\sigma$, where
\begin{equation}\label{root_reduction_eqn}
\bar\sigma(\bar a)=\overline{\sigma(a)}\;\;\text{for every $\sigma\in G_{\P}$ and $a\in\calO_N$}.
\end{equation}

The condition $\Delta(c)\ne 0$ implies that $\p_c$ is unramified in $N$, so this map is in fact an isomorphism. If $x_1,\ldots, x_n\in\calO_N$ denote the roots of $P$ in $N$, then $\bar x_1,\ldots, \bar x_n$ are the roots of $P_c$ in $\kappa(\P)$. Moreover, one can show that
\[\kappa(\P)=\Q(\bar x_1,\ldots, \bar x_n),\]

so that $\kappa(\P)$ is a splitting field for $P_c$ and hence $G_c=\Gal(\kappa(\P)/\Q)$. We therefore have an isomorphism $G_{\P}\to G_c$ and a bijection $x_i\mapsto\bar x_i$ between the set of roots of $P$ and the set of roots of $P_c$. The property \eqref{root_reduction_eqn} then implies that $G_{\P}\equiv G_c$, proving (1).

Suppose now that $G_c\cong G$. To prove (2) we will show that for every irreducible factor $Q\in\Q[t][x]$ of $P$, the polynomial $Q(c,x)$ remains irreducible; the result then follows immediately.

Let $Q$ be an irreducible factor of $P$ and let $\P$ be a prime of $N$ dividing $\p_c$. By (1) we have $G_{\P}\equiv G_c$. Since $G_c\cong G$ and $G_c\equiv G_{\P}$, we must have $G_{\P}=G$ and therefore $G_c\equiv G$. The proof of (1) provides an isomorphism $G\to G_c$ and a bijection $r$ between the roots of $P$ and the roots of $P_c$. The map $r$ is the restriction of a ring homomorphism, so the roots of $Q$ are mapped bijectively by $r$ to the roots of $Q(c,x)$. Thus it is still the case that $G\equiv G_c$ if we regard $G$ as acting on the roots of $Q$ and $G_c$ as acting on the roots of $Q(c,x)$. Since $G$ acts transitively on the roots of $Q$, $G_c$ must act transitively on the roots of $Q(c,x)$. Hence $Q(c,x)$ is irreducible, as claimed.
\end{proof}

The following two propositions will be our main tools for describing the exceptional set of $P$ and determining the Galois group of every exceptional specialization $P_c$.

\begin{prop}\label{HIT_subgroup_curve_prop}
Let $H$ be a subgroup of $G$ and let $f(t,x)$ be a monic irreducible polynomial in $A[x]$ such that the fixed field of $H$ is generated over $\Q(t)$ by a root of $f$. Suppose that $c\in\Q$ satisfies
\[\Delta(c)\cdot\ell(c)\cdot\disc f(c,x)\ne 0.\]

Then the following are equivalent:
\begin{enumerate}
\item The polynomial $f(c,x)$ has a rational root.
\item There exists a prime $\P$ of $N$ dividing $\p_c$ such that $G_{\P}\subseteq H$.
\end{enumerate}
\end{prop}
\begin{proof}
Let $K$ be the fixed field of $H$ and let $\theta$ be a root of $f(t,x)$ such that $K=\Q(t)(\theta)$. Let $\calO_K$ be the integral closure of $A$ in $K$. The condition $\disc f(c,x)\ne 0$ implies that $\p_c$ is relatively prime to the conductor of the ring $A[\theta]$ in $\calO_K$. The Dedekind-Kummer theorem \cite[p. 47, Prop. 8.3]{neukirch} therefore allows us to relate the factorization of the polynomial $f(c,x)$ to the factorization of $\p_c$ in $\calO_K$. In particular, the theorem implies that $f(c,x)$ has a rational root if and only if there exists a prime $\q$ of $K$ dividing $\p_c$ such that the residue degree of $\q$ over $\p_c$ is equal to 1. Since $\q$ is unramified over $\p_c$, a maximality property of decomposition fields (see \cite[p. 118, Prop. 8.6]{lorenzini}) implies that (1) occurs if and only if there exists a prime $\P$ of $N$ dividing $\p_c$ such that $K$ is contained in the fixed field of $G_{\P}$. This proves the proposition, since the latter condition is clearly equivalent to (2).
\end{proof}

\begin{prop}\label{HIT_exceptional_set_prop}
Let $M_1,\ldots, M_r$ be representatives of all the conjugacy classes of maximal subgroups of $G$. For $i=1,\ldots, r$ let $F_i$ be the fixed field of $M_i$ and let $f_i\in\Q[t][x]$ be a monic irreducible polynomial such that the extension $F_i/\Q(t)$ is generated by a root of $f_i$. Suppose that $c\in\Q$ satisfies
\[\Delta(c)\cdot\ell(c)\cdot\prod_{i=1}^r\disc f_i(c,x)\ne 0.\]
Then $c\in\calE(P)$ if and only if there is an index $i$ such that $f_i(c,x)$ has a rational root.
\end{prop}
\begin{proof}
Suppose first that $f_i(c,x)$ has a rational root for some $i$. Then Proposition \ref{HIT_subgroup_curve_prop} implies that there exists a prime $\P$ of $N$ dividing $\p_c$ such that $G_{\P}\subseteq M_i$. By Lemma \ref{HIT_lem} we know that $G_c$ is isomorphic to $G_{\P}$, which is a proper subgroup of $G$. Thus $G_c$ cannot be isomorphic to $G$, so $c\in\calE(P)$.

Conversely, suppose that $c\in\calE(P)$. Then $G_c\not\cong G$ by part (2) of Lemma \ref{HIT_lem}. Let $\P$ be a prime of $N$ dividing $\p_c$, so that $G_c\cong G_{\P}$. Since $G_c\not\cong G$, $G_{\P}$ must be a proper subgroup of $G$. Replacing $\P$ by a conjugate ideal if necessary, we then have $G_{\P}\subseteq M_i$ for some $i$. By Proposition \ref{HIT_subgroup_curve_prop} this implies that $f_i(c,x)$ has a rational root.
\end{proof}

\begin{rem}[Algorithmic aspects]\label{HIT_algorithmic_rem}
Propositions \ref{HIT_subgroup_curve_prop} and \ref{HIT_exceptional_set_prop} can be used in practice to explicitly describe the exceptional set of $P$ and to determine the structure of the Galois group $G_c$ for every $c\in\calE(P)$. It is crucial for these purposes to be able to compute the Galois group $G$ as well as a defining polynomial for the fixed field of any subgroup of $G$. Both of these tasks can be carried out using known methods. Indeed, an algorithm of Fieker and Kl\"uners \cite{fieker-kluners} can be used to compute $G$, and then the fixed field of any subgroup of $G$ can be computed as described in \cite[\S3.3]{kluners-malle}, for instance. In the case where $P$ is irreducible, these methods were implemented by Fieker and are available via the \textsc{Magma} \cite{magma} functions \texttt{GaloisGroup} and \texttt{GaloisSubgroup}. The general case where $P$ is allowed to be reducible is treated in the article \cite{krumm-sutherland}; however we will not be needing here the algorithms for that case.
\end{rem}

We end this section with a proposition that can be used to determine the factorization types $\calF(P_c)$ for all $c\in\calE(P)$ when the set $\calE(P)$ is a finite union of infinite sets, each of which can be parametrized by a rational function. As will be shown in Proposition \ref{Phi4_exceptional_set_prop}, this is the case for the exceptional set of the dynatomic polynomial $\Phi_4(t,x)$.

\begin{prop}\label{rational_function_galois_group_prop}
Suppose that $\mu(z)=a(z)/b(z)\in\Q(z)$ is a rational function such that the polynomial $f(t,x)=t\cdot b(x)-a(x)\in\Q[t][x]$ is monic and irreducible and has a root $\theta\in N$. Let $L=\Q(t)(\theta)$ and $H=\Gal(N/L)$. Then the following hold:
\begin{enumerate}
\item Suppose that $c\in\Q$ satisfies $\Delta(c)\cdot\ell(c)\cdot\disc f(c,x)\ne 0$ and has the form $c=\mu(v)$ for some $v\in\Q$. Then there exists a prime $\P$ of $N$ dividing $\p_c$ such that $G_{\P}\subseteq H$.
\item Define a polynomial $g(t,x)=b(t)^m\cdot P(\mu(t),x)$, where $m\ge 0$ is chosen so that $g$ has coefficients in $\Q[t]$. Then $\Gal(g)\equiv H$.
\item With notation and assumptions as above, suppose furthermore that $G_{\P}=H$ and $\disc g(v,x)\ne 0$. Then $G_c\equiv H$ and $\calF(P_c)=\calF(g)$.
\end{enumerate}
\end{prop}
\begin{proof}
With assumptions as in (1), the polynomial $f(c,x)$ has a rational root (namely $v$), so Proposition \ref{HIT_subgroup_curve_prop} implies that there exists a prime $\P$ of $N$ dividing $\p_c$ such that $G_{\P}\subseteq H$. This proves (1).

From the assumption that $f$ is irreducible it follows that $b(\theta)\ne 0$, so $\mu(\theta)$ is defined and $\mu(\theta)=t$. In particular $t\in\Q(\theta)$, so $L=\Q(\theta)$. Clearly then $\theta$ is transcendental over $\Q$, so the map $t\mapsto\theta$ is an isomorphism $\sigma:\Q(t)\to L$.

We can now prove (2). Let $E/\Q(t)$ be a splitting field for $g(t,x)$. Applying $\sigma$ to the coefficients of $g$ we obtain the polynomial $g(\theta, x)\in L[x]$. Note that since $\mu(\theta)=t$, then $g(\theta,x)=b(\theta)^m\cdot P(t,x)$, so $N$ is a splitting field for $g(\theta,x)$. Thus, by basic field theory, $\sigma$ extends to an isomorphism $E\to N$. From here it follows easily that $\Gal(g)\equiv\Gal(N/L)=H$, proving (2). 

Finally, we prove (3). Note that $g(v,x)=b(v)^m\cdot P(c,x)$, so that the polynomials $g(v,x)$ and $P_c$ have the same Galois group and factorization type. Applying part 1 of Lemma \ref{HIT_lem} and using (2) we obtain
\begin{equation}\label{parametrizable_exc_set_galois_eq}
\Gal(g(v,x))=G_c\equiv G_{\P}=H\equiv\Gal(g).
\end{equation}
In particular $G_c\equiv H$, which proves the first part of (3). To prove the second part we apply part 2 of Lemma \ref{HIT_lem} to $g$ and $v$ instead of $P$ and $c$. The leading coefficient and the discriminant of $g(v,x)$ are both nonzero, and we have $\Gal(g(v,x))\cong\Gal(g)$ by \eqref{parametrizable_exc_set_galois_eq}, so the lemma implies that $\calF(g(v,x))=\calF(g)$. Since $g(v,x)$ and $P_c$ have the same factorization type, we conclude that $\calF(P_c)=\calF(g)$, which completes the proof of (3).
\end{proof}

\section{Specializations of $\Phi_4$}\label{phi4_factorization_galois_section}

The polynomial $\Phi_4(t,x)$ is given by
\begin{multline*}\label{phi4_expression}
\Phi_4(t,x)=x^{12} + (6 t) x^{10} + x^{9} + (15 t^{2} + 3 t) x^{8} + (4 t) x^{7} + (20 t^{3} + 12 t^{2} + 1) x^{6}\\ + (6 t^{2} + 2 t) x^{5} + (15 t^{4} + 18 t^{3} + 3 t^{2} + 4 t) x^{4} + (4 t^{3} + 4 t^{2} + 1) x^{3}+\\ (6 t^{5} + 12 t^{4} + 6 t^{3} + 5 t^{2} + t) x^{2} + (t^{4} + 2 t^{3} + t^{2} + 2 t) x + t^{6} + 3 t^{5} + 3 t^{4} + 3 t^{3} + 2 t^{2} + 1.
\end{multline*}

Our goal in this section is to classify the Galois groups and factorization types of the specialized polynomials $\Phi_4(c,x)$ as $c$ varies in $\Q$; the main result in this direction is Theorem \ref{phi4_factorization_galois_thm}.

Let $N/\Q(t)$ be a splitting field of $\Phi_4$ and let $G=\Gal(N/\Q(t))$ be the Galois group of $\Phi_4$.  We begin in \S\ref{subgroup_lattice_section} by making some observations regarding the lattice of subgroups of $G$. Then in \S\ref{subgroup_curves_section} we determine the rational points on several curves corresponding to subgroups of $G$. Finally, the desired classification theorem is proved in \S\ref{group_classification_section}.

\subsection{Subgroups of $G$}\label{subgroup_lattice_section}

As noted in \S\ref{dynatomic_section}, the group $G$ is isomorphic to the wreath product $(\Z/4\Z)\wr S_3$. We will need to use the action of $G$ on the roots of $\Phi_4$, so we choose a particular permutation representation of $G$.

The map $f(x)=x^2+t$ permutes the roots of $\Phi_4$, so the set of roots can be partitioned into $f$-orbits. The fact that $\Phi_4$ is irreducible and hence separable implies that every orbit has size four (see \cite[Thm. 2.4(c)]{morton-patel}). The roots of $\Phi_4$ are thus partitioned into three orbits of size four. Labeling the roots as in Figure \ref{roots_labeling_figure} below, we obtain an embedding $G\into S_{12}$ whose image is the centralizer of the permutation $(1,2,3,4)(5,6,7,8)(9,10,11,12)$; see \cite[pp. 2275-2277]{krumm_lgp}. Denoting this centralizer by $\calW$ we therefore have $G\equiv\calW$. We will henceforth identify $G$ with $\calW$.

\begin{figure}[!h]
\includegraphics[scale=0.9]{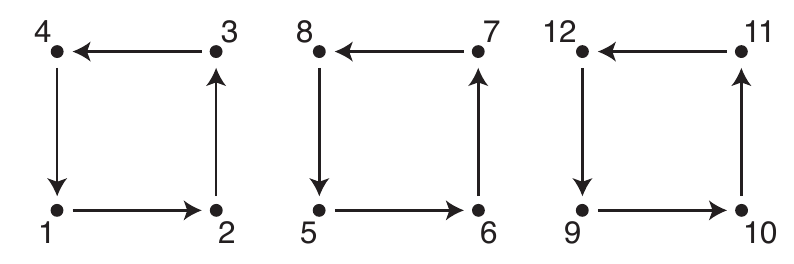}
\vspace*{8pt}
\caption{Labeling of the roots of $\Phi_{4}(t,x)$. Vertices represent the roots of $\Phi_4$, and directed edges represent the action of the map $x^2+t$.}
\label{roots_labeling_figure}
\end{figure}

Up to conjugacy, $G$ has exactly five maximal subgroups, which we denote by $M_1,\ldots, M_5$. We choose the labeling of these groups according to the following conditions, which determine the groups uniquely up to conjugation.
\[M_1\cong (192,944),\; |M_2|=128,\; |M_3|=96,\; M_4\cong(192,188),\; M_5\cong(192,182).\]

Here, the notation $M\cong (o,n)$ means that $M$ is isomorphic to the group labeled $(o,n)$ in the  \textsc{Small Groups} library \cite{small_groups}, which is distributed with the \textsc{Gap} \cite{gap} and \textsc{Magma} \cite{magma} computer algebra systems. This implies, in particular, that $o$ is the order of $M$.

In addition to the groups $M_i$, we will need to work with the maximal subgroups of these groups. The group $M_1$ has five maximal subgroups up to conjugation, which we denote by $A_1,\ldots, A_5$. In order to uniquely identify these groups it does not suffice to consider only their orders or their isomorphism types. Thus, we will refer to properties of the cycle decompositions of their elements. The labeling of the $A_i$ is chosen according to the following conditions:

\begin{itemize}[leftmargin=5mm]
\item $|A_1|=48,\;|A_2|=64$, $A_3\cong(96,68)$;
\item $A_4\cong (96,64)$ and has an element whose disjoint cycle decomposition is a product of six 2-cycles;
\item $A_5\cong(96,64)$ and does not have an element whose disjoint cycle decomposition is a product of six 2-cycles.
\end{itemize}

The group $M_2$ has seven maximal subgroups up to conjugation, which we denote by $B_1,\ldots, B_7$. We choose the labeling as follows:
\begin{itemize}[leftmargin=5mm]
\item $B_1\cong(64,55)$, $B_2\cong(64,85)$, $B_3\cong(64,198)$;
\item $B_4\cong(64,20)$ and has an element with cycle decomposition of the form $(8\text{-cycle})(4\text{-cycle})$;
\item $B_5\cong (64,20)$ and does not have an element with cycle decomposition of the form $(8\text{-cycle})(4\text{-cycle})$;
\item $B_6\cong(64,101)$ and has an element with cycle decomposition of the form $(8\text{-cycle})(4\text{-cycle})$;
\item $B_7\cong(64,101)$ and does not have an element with cycle decomposition of the form $(8\text{-cycle})(4\text{-cycle})$.
\end{itemize}

The group $M_3$ has five maximal subgroups up to conjugation, which we denote by $C_1,\ldots, C_5$. We label these subgroups according to the conditions
\[|C_1|=24,\; |C_2|=32,\; C_3\cong (48,30),\; C_4\cong (48,31),\; C_5\cong (48,48).\]

The group $M_4$ has three maximal subgroups up to conjugation, which we denote by $D_1,D_2,D_3$. The labeling is chosen as follows:
\[|D_1|=48,\;|D_2|=64,\;|D_3|=96.\]

Finally, the group $M_5$ has three maximal subgroups up to conjugation, which we denote by $E_1,E_2,E_3$ and label according to the properties
\[|E_1|=48,\;|E_2|=64,\;|E_3|=96.\]

Some of the groups listed above are conjugate in $G$. In particular, we note the following pairs of conjugate subgroups of $G$:
\begin{equation}\label{conjugate_pairs}
(A_1,C_5),\;(A_2,B_6),\;(A_3,D_3),\;(B_1,D_2),\;(B_5,E_2).
\end{equation}

We summarize some of the above information in Figure~\ref{subgroup_lattice_figure}.

\begin{figure}[!h]
\includegraphics[scale=0.23]{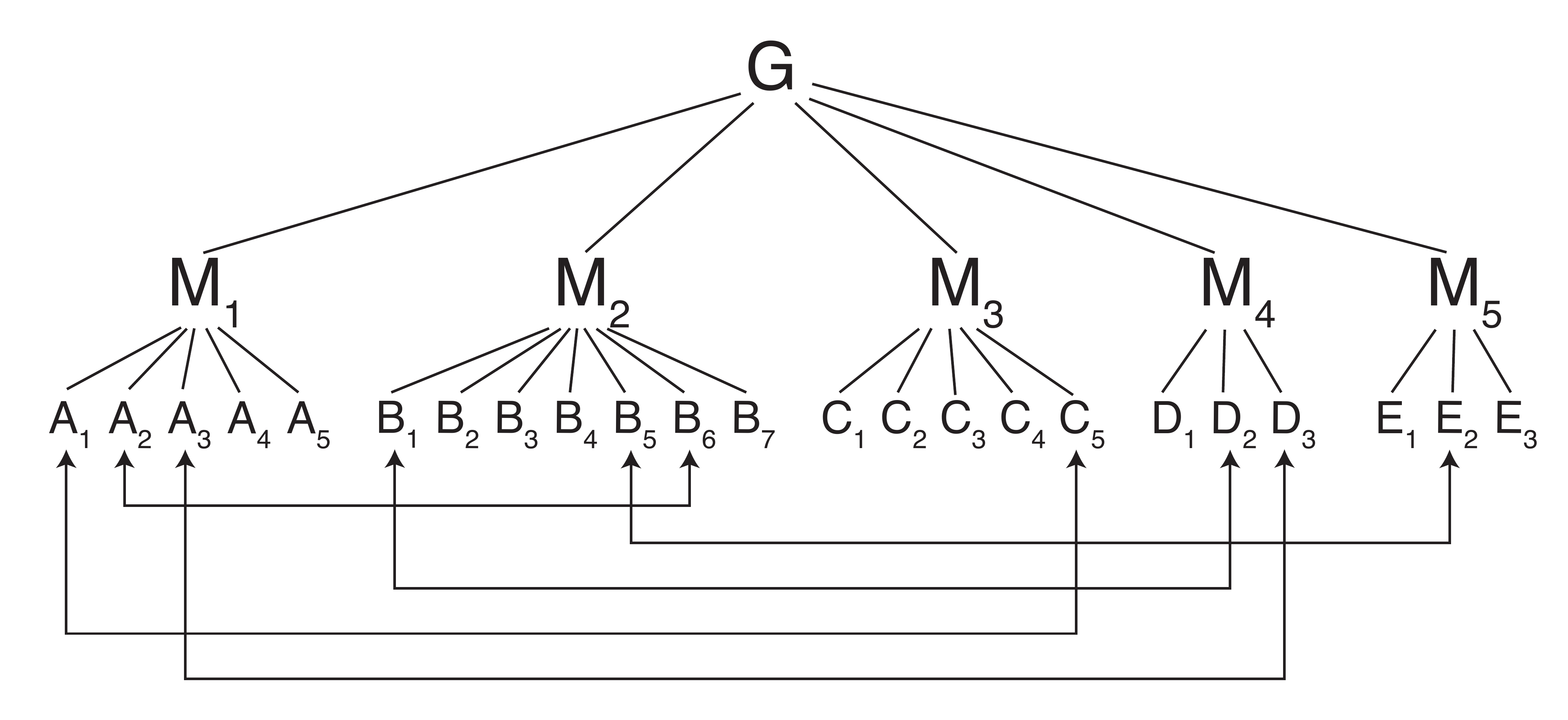}
\vspace*{8pt}
\caption{Part of the lattice of subgroups of $G$. An arrow joining two subgroups indicates conjugacy in $G$.}
\label{subgroup_lattice_figure}
\end{figure}

The subgroups of $\calW$ corresponding to $M_1$, $M_2$, and $B_7$ will be especially important in our analysis in \S\ref{group_classification_section}. These three groups are conjugate in $\calW$ to the groups $\calG$, $\calH$, and $\calI$ defined in Appendix \ref{galois_data_appendix}. Thus we have
\begin{equation}\label{subgroup_equivalence}
G\equiv\calW,\; M_1\equiv\calG,\; M_2\equiv\calH,\; B_7\equiv\calI.
\end{equation}

\subsection{Curves corresponding to subgroups of $G$}\label{subgroup_curves_section}

To every subgroup $H$ of $G$ we associate an affine algebraic curve as follows. Let $L$ be the fixed field of $H$, let $\theta$ be a primitive element for the extension $L/\Q(t)$ such that $\theta$ is integral over $\Q[t]$, and let $f(t,x)\in\Q[t][x]$ be the minimal polynomial of $\theta$ over $\Q(t)$. Then the curve corresponding to $H$ is given by the equation $f(t,x)=0$. In practice, the polynomial $f$ can be computed using the methods discussed in Remark \ref{HIT_algorithmic_rem}.

In this section we will determine the rational points on the curves corresponding to several of the subgroups of $G$ defined in \S\ref{subgroup_lattice_section}. All of the curves we need to consider are either rational, elliptic, or hyperelliptic. An important tool in what follows is an algorithm for computing a parametrization of a given rational curve; we refer the reader to \cite[Chap. 4]{sendra-winkler-perez} for a discussion of the algorithmic aspects of this problem. For our computations we have used the \texttt{Parametrization} function in \textsc{Magma}. All of the elliptic curves that arise have rank 0, so it is a straightforward calculation to determine their rational points. For the hyperelliptic curves, the methods we use to find their rational points are discussed in \S\ref{rational_points_section}. 

We begin by considering the curves corresponding to the maximal subgroups $M_i$. Computing the fixed fields of these groups we obtain the following polynomials:
\begin{align*}
f_1(t,x) &= x^2 - 64t^2 + 32t - 20,\\
f_2(t,x) &= x^3 + (4t + 3)x + 4,\\
f_3(t,x) &= x^4 + (12 - 16t)x^3 + (64t^2 - 64t + 36)x^2 + (256t^2 - 64t + 160)x + 256t^2,\\
f_4(t,x) &= x^2 + 256t^3 + 576t^2 + 432t + 540,\\
f_5(t,x) &= x^2 + 16(16t^2 - 8t + 5)(64t^3 + 144t^2 + 108t + 135).\\
\end{align*}

\vspace{-4mm}
The curve associated to $M_i$ is given by the equation $f_i(t,x)=0$. We will now determine the rational points on these five curves.

\begin{lem}\label{maximal_subgroup_curves_lem}
Let $\calC_i$ denote the affine plane curve defined by $f_i(t,x)=0$. Let $\lambda$ and $\eta$ be the rational functions defined in \eqref{rational_fncn_def}. Then the sets of rational points on the curves $\calC_i$ are given by
\begin{align}
\calC_1(\Q)&=\{(\lambda(z),(z^2+4)/z): z\in\Q^{\ast}\},\label{f1_curve_points}\\
\calC_2(\Q)&=\{(\eta(z),-z):z\in\Q^{\ast}\},\label{f2_curve_points}\\
\calC_3(\Q)&=\{(0,0), (0, -10), (-5/2,-10)\}, \label{f3_curve_points}\\
\calC_4(\Q)&=\calC_5(\Q)=\emptyset\label{f4,5_curve_points}.
\end{align}
\end{lem}
\begin{proof}
The curve $\calC_1$ is parametrizable. Indeed, the rational maps
\[\phi:\calC_1\to\A^1=\Spec\Q[z]\;\;\text{and}\;\; \psi:\A^1\dashedrightarrow\calC_1\]
given by $\phi(t,x)=(8t+x-2)/2$ and $\psi(z)=\left(\lambda(z),(z^2+4)/z\right)$ are inverses. By a simple calculation we find that the pullback of 0 under $\phi$ contains no rational point of $\calC_1$; \eqref{f1_curve_points} now follows easily.

The curve $\calC_2$ is also parametrizable. Solving for $t$ in the equation $f_2(t,x)=0$ we obtain $t=\eta(-x)$; this implies \eqref{f2_curve_points}.

Let $E$ be the elliptic curve $Y^2 = X^3 + X^2 + 4X + 4$. There is a birational map $\phi:\calC_3\dashedrightarrow E$ given by $\phi=(\phi_1,\phi_2)$, where
\[\phi_1(t,x) = \frac{-x-10}{x}\;\;\;\text{and}\;\;\;\phi_2(t,x) = \frac{5x^2 + (10-40t)x - 80t}{2x^2}.\]
The curve $E$ has Cremona label 20a1. This curve has rank 0, and its affine rational points are $(-1,0)$, $(0,\pm 2)$, and $(4,\pm 10)$. Note that any rational point on $\calC_3$ different from $(0,0)$ is mapped by $\phi$ to one of these five points. Pulling back all the affine rational points on $E$ we obtain \eqref{f3_curve_points}.

The projective closure of $\calC_4$ is the elliptic curve with Cremona label 108a2, which has a trivial Mordell-Weil group. It follows that $\calC_4(\Q)=\emptyset$.

The curve $\calC_5$ is an affine model for a hyperelliptic curve $\calX$ whose Jacobian has rank 0. Computing the rational points on $\calX$ we find that $\calX(\Q)=\{\infty\}$. Hence $\calC_5(\Q)=\emptyset$, and \eqref{f4,5_curve_points} is proved.
\end{proof}

Next we will determine the rational points on the curves corresponding to the groups $A_4$, $A_5$, $B_2$, $B_3$, $B_4$, and $B_7$. The fixed fields of these subgroups of $G$ are generated by the polynomials $f(t,x)$ defined in Lemmas \ref{A4_curve_lem}-\ref{B7_curve_lem}. Hence the curves $\calA_4,\calA_5,\calB_2$, etc. defined in these lemmas are the curves corresponding to the above subgroups.

\begin{lem}\label{A4_curve_lem} Let $f(t,x)=x^4-p(t)x^2+q(t)$, where
\begin{align*}
p(t) &= 32(4t + 5)(16t^2 - 8t + 5)(64t^3 + 144t^2 + 108t + 135),\\
q(t) &= 256(4t - 1)^2(4t + 5)^2(16t^2 - 8t + 5)(64t^3 + 144t^2 + 108t + 135)^2.
\end{align*}
Let $\calA_4$ be the affine curve defined by the equation $f(t,x)=0$. Then
\[\calA_4(\Q)=\{(-5/4,0), (1/4,0)\}.\]
\end{lem}
\begin{proof}
Let $X$ be the curve in $\A^3=\Spec\Q[t,x,y]$ defined by the equations
\[
\begin{cases}
2x^2=p(t) + y,\\
y^2=p(t)^2-4q(t).
\end{cases}
\]
Note that there is a map $X\to \calA_4$ given by $(t,x,y)\mapsto(t,x)$, and that every rational point on $\calA_4$ lifts to a rational point on $X$. 

Suppose that $(t_0,x_0)\in\calA_4(\Q)$ and let $y_0$ be a rational number such that $(t_0,x_0,y_0)\in X(\Q)$. If $t_0=-5/4$, then $p(t_0)=q(t_0)=0$, so the equation $f(t_0,x_0)=0$ implies that $x_0=0$. Hence $(t_0,x_0)=(-5/4,0)\in\calA_4(\Q)$.

 Assuming now that $t_0\ne-5/4$, we may define
\[z_0=y_0/\left(2^6(4t_0+5)(64t_0^3 + 144t_0^2 + 108t_0 + 135)\right).\]

It follows from the equation $y_0^2=p(t_0)^2-4q(t_0)$ and the factorization
\[p(t)^2-4q(t)=2^{12}(4t+5)^2(64t^3 + 144t^2 + 108t + 135)^2(16t^2 - 8t + 5)\]

that $z_0^2=16t_0^2 - 8t_0 + 5$. Let $D$ be the affine curve defined by
\[z^2=16t^2-8t+5.\]

The curve $D$ is parametrizable. Indeed, there are inverse rational maps $\theta:D\dashedrightarrow\A^1=\Spec\Q[v]$ and $\psi:\A^1\dashedrightarrow D$ given by
\[\theta(t,z)=\frac{z+2}{4t-1}\;\;\text{and}\;\;\psi(v)=\left(\frac{v^2+4v-1}{4(v^2-1)},\frac{2v^2+2}{v^2-1}\right).\]

It is now a straightforward calculation to verify that
\begin{equation}\label{A4_D_curve_points}
D(\Q)=\{(1/4,2)\}\cup\{\psi(v):v\in\Q\setminus\{\pm 1\}\}.
\end{equation}

Since $(t_0,z_0)$ is a rational point on $D$, then either $(t_0,z_0)=(1/4,2)$ or there exists $v\in\Q\setminus\{\pm 1\}$ such that $(t_0,z_0)=\psi(v)$. In the former case the equation $f(t_0,x_0)=0$ becomes $x_0^2(x_0^2 - 132096)=0$, so $x_0=0$. Thus we have found the point $(t_0,x_0)=(1/4,0)\in\calA_4(\Q)$.

Assume now that $(t_0,z_0)=\psi(v)$ for some $v\in\Q\setminus\{\pm 1\}$. If $v=0$, then $(t_0,z_0)=(1/4,-2)$ and we again obtain the point $(t_0,x_0)=(1/4,0)$. If $v\ne 0$, then expressing $t_0, z_0$, and $y_0$ in terms of $v$, the equation $2x_0^2=p(t_0)+y_0$ implies that
\[x_0^2(v^2-1)^6=2^{10}v^2(v^2+1)(3v^2 + 2v - 3)(43v^6 + 48v^5 - 81v^4 - 80v^3 + 81v^2 + 48v - 43).\]

 Setting $w=x_0(v^2-1)^3/(32v)$ we arrive at the equation
\[w^2=(v^2+1)(3v^2 + 2v - 3)(43v^6 + 48v^5 - 81v^4 - 80v^3 + 81v^2 + 48v - 43).\]

By Proposition \ref{genus4_curve} this implies that $v=\pm 1$, which is a contradiction. Hence the points $(-5/4,0)$ and $(1/4,0)$ are the only rational points on $\calA_4$.
\end{proof}

\begin{lem}\label{A5_curve_lem}
Let $f(t,x)=x^4+p(t)x^2+q(t)$, where
\begin{align*}
p(t) &= 2(4t + 5)(16t^2 - 8t + 5),\\
q(t) &= (4t - 1)^2(4t + 5)^2(16t^2 - 8t + 5).
\end{align*}
Let $\calA_5$ be the affine curve defined by the equation $f(t,x)=0$. Then
\[\calA_5(\Q)=\{(-5/4,0),(1/4,0)\}.\]
\end{lem}
\begin{proof}
Let $X$ be the curve in $\A^3=\Spec\Q[t,x,y]$ defined by the equations
\[
\begin{cases}
2x^2=y-p(t),\\
y^2=p(t)^2-4q(t).
\end{cases}
\]

There is a projection map $X\to \calA_5$ given by $(t,x,y)\mapsto(t,x)$, and every rational point on $\calA_5$ lifts to a rational point on $X$. 

Suppose that $(t_0,x_0)\in\calA_5(\Q)$ lifts to $(t_0,x_0,y_0)\in X(\Q)$.  If $t_0=-5/4$, then $p(t_0)=q(t_0)=0$, so the equation $f(t_0,x_0)=0$ implies that $x_0=0$. Thus we have found the point $(-5/4,0)\in\calA_5(\Q)$.

 Assuming now that $t_0\ne-5/4$, we may define $z_0=y_0/\left(4(4t_0+5)\right)$.

From the equation $y_0^2=p(t_0)^2-4q(t_0)$ and the factorization
\[p(t)^2-4q(t)=16(4t+5)^2(16t^2 - 8t + 5)\]

it follows that $z_0^2=16t_0^2 - 8t_0 + 5$. Thus $(t_0,z_0)\in D(\Q)$, where $D$ is the curve defined in the proof of Lemma \ref{A4_curve_lem}. By \eqref{A4_D_curve_points}, either $(t_0,z_0)=(1/4,2)$ or $(t_0,z_0)=\psi(v)$ for some $v\ne\pm 1$. In the first case the equation $f(t_0,x_0)=0$ becomes $x_0^2(x_0^2+48)=0$, so $x_0=0$. This yields the point $(1/4,0)\in\calA_5(\Q)$.

Suppose now that $(t_0,z_0)=\psi(v)$ for some $v\ne\pm 1$. Expressing $t_0$, $z_0$, and $y_0$ in terms of $v$, the equation $2x_0^2=y_0-p(t_0)$ implies that
\[x_0^2(v^2 - 1)^4 =-16(v^2-1)(v^2 + 1)(3v^2 + 2v - 3).\]

Letting $w=x_0(v^2-1)^2/4$ we obtain
\[w^2 =-(v^2-1)(v^2 + 1)(3v^2 + 2v - 3).\]

The above equation defines a hyperelliptic curve whose Jacobian has rank 0. Computing the rational points on the curve we obtain only the two points $(\pm 1, 0)$. It follows that $v=\pm 1$, which is a contradiction. Hence the only rational points on $\calA_5$ are $(-5/4,0)$ and $(1/4,0)$.
\end{proof}

\begin{lem}\label{B2_curve_lem}
Let $f(t,x)=x^6+p(t)x^4-q(t)x^2+r(t)$, where
\begin{align*}
p(t) &= 4(64t^3 - 16t^2 - 36t + 45),\\
q(t) &= 16(4t - 3)(16t^2 - 8t + 5)(32t^2 + 60t + 45),\\
r(t) &= 64(16t^2 - 8t + 5)^2(64t^3 + 144t^2 + 108t + 135).
\end{align*}
Let $\calB_2$ be the affine curve defined by the equation $f(t,x)=0$. Then
\[\calB_2(\Q)=\{(-2,\pm 34)\}.\]
\end{lem}
\begin{proof}
Let $C$ be the curve in $\A^2=\Spec\Q[t,y]$ defined by the equation
\[y^3+p(t)y^2-q(t)y+r(t)=0.\]

Note that there is a map $\calB_2\to C$, $(t,x)\mapsto(t,x^2)$. The curve $C$ is parametrizable; a birational map $\A^1\dashedrightarrow C$ is given by 
\[\psi(v)=\left(-(v^3 + 3v^2 + 16)/(4v^2), 4(v^3 + 4)(v^4 + 8v^3 + 16v^2 + 64)/v^4\right).\]
Computing an inverse $\theta$ of $\psi$ we find that $\theta$ is defined at every rational point on $C$ different from $(0,-60)$, and that the pullback of 0 under $\theta$ contains no rational point on $C$.

Suppose that $(t_0,x_0)\in\calB_2(\Q)$. Clearly $(t_0,x_0^2)\ne (0,-60)$, so we may define $v=\theta(t_0,x_0^2)$. Since $v\ne 0$, we have $(t_0,x_0^2)=\psi(v)$; in particular
\[x_0^2=4(v^3 + 4)(v^4 + 8v^3 + 16v^2 + 64)/v^4.\]

Letting $w=v^2x_0/2$ we obtain $w^2=(v^3 + 4)(v^4 + 8v^3 + 16v^2 + 64)$. By Proposition \ref{genus3_chabauty} this implies that $v=4$, so $(t_0,x_0^2)=\psi(v)=(-2, 1156)$ and therefore $(t_0,x_0)=(-2,\pm 34)$. Hence $\calB_2(\Q)=\{(-2,\pm 34)\}$.
\end{proof}

\begin{lem}\label{B3_curve_lem}
Let $f(t,x)=x^6-p(t)x^4-q(t)x^2-r(t)$, where
\begin{align*}
p(t) &= (4t - 1)(4t - 9),\\
q(t) &= 8(4t - 3)(16t^2 - 8t + 5),\\
r(t) &= 16(16t^2 - 8t + 5)^2.
\end{align*}
Let $\calB_3$ be the affine curve defined by the equation $f(t,x)=0$. Then
\[\calB_3(\Q)=\{(0,\pm 2)\}.\]
\end{lem}
\begin{proof}
Let $C$ be the curve in $\A^2=\Spec\Q[t,y]$ defined by the equation
\[y^3-p(t)y^2-q(t)y-r(t)=0.\]

There is a birational map $\psi:\A^1\dashedrightarrow C$ given by 
\[\psi(v)=\left(-\frac{(v + 4)(5v^2 + 25v + 32)}{4(v + 2)^2(v + 3)}, \frac{4(5v^4 + 60v^3 + 268v^2 + 528v + 388)}{(v + 2)^2(v + 3)^2}\right).\]

Computing an inverse for $\psi$ we obtain the map $\theta:C\dashedrightarrow\A^1$ defined by
\[\theta(t,y)=\frac{y^2 - (16t^2 - 24t - 35)y - 256t^3 - 64t^2 + 16t - 60}{8(16t^2 - 8t + 5 - 2y)}.\]

Outside the set $\{(-5/4, 20), (-9/4, 52)\}$, $\theta$ is defined at every rational point on $C$, and there is no rational point on $C$ mapped by $\theta$ to $-2$ or $-3$. 

Suppose that $(t_0,x_0)\in\calB_3(\Q)$. Then $(t_0,x_0^2)\in C(\Q)$ and clearly
\[(t_0,x_0^2)\notin\{(-5/4, 20), (-9/4, 52)\},\]

so we may define $v=\theta(t_0,x_0^2)$ and then $v\ne-2,-3$. Hence $(t_0,x_0^2)=\psi(v)$; in particular
\[x_0^2=\frac{4(5v^4 + 60v^3 + 268v^2 + 528v + 388)}{(v + 2)^2(v + 3)^2}.\]

Letting $w=x_0(v+2)(v+3)/2$ we obtain
\[w^2=5v^4 + 60v^3 + 268v^2 + 528v + 388.\]

The hyperelliptic curve $X$ defined by this equation is isomorphic to the elliptic curve with Cremona label 20a2 and therefore has exactly six rational points. A search for points on $X$ yields
\[X(\Q)=\{(-2,\pm 2), (-3,\pm 1), (-4,\pm 2)\}.\]

It follows that $v=-4$, so $(t_0,x_0^2)=\psi(-4)=(0,4)$ and therefore $(t_0,x_0)=(0,\pm 2)$. Hence $\calB_3(\Q)=\{(0,\pm 2)\}$.
\end{proof}

\begin{lem}\label{B4_curve_lem}
Let $f(t,x)=x^6-p(t)x^4+q(t)x^2+r(t)$, where
\begin{align*}
p(t) &= 2(4t - 9)(4t + 3),\\
q(t) &= 256t^4 - 1792t^3 - 2592t^2 - 432t - 1431,\\
r(t) &= 16(16t^2 - 8t + 5)(64t^3 + 144t^2 + 108t + 135).
\end{align*}
Let $\calB_4$ be the affine curve defined by the equation $f(t,x)=0$. Then
\[\calB_4(\Q)=\{(-5,\pm 20)\}.\]
\end{lem}
\begin{proof}
Let $C$ be the curve in $\A^2=\Spec\Q[t,y]$ defined by the equation
\[y^3-p(t)y^2+q(t)y+r(t)=0.\]

The curve $C$ is parametrizable; a pair of inverse maps $\theta:C\dashedrightarrow\A^1$ and $\psi:\A^1\dashedrightarrow C$ is given by
\[\theta(t,y)=\frac{(4t + 9)y - 64t^3 - 240t^2 - 252t + 63}{2(16t^2 - 72t + 81 - 3y)},\]

\[\psi(v)=\left(\frac{-(v + 1)(2v^2 - 5v + 5)}{4(v - 1)^2},\frac{16(2v^3 - 6v^2 + 6v - 1)(v^2 - 2v + 2)}{(v - 1)^4}\right).\]

We find that $\theta$ is defined at every rational point on $C$, and there is no such point mapping to $v=1$. Hence every rational point on $C$ has the form $\psi(v)$ for some $v\in\Q$.

Suppose that $(t_0,x_0)\in\calB_4(\Q)$. Then $(t_0,x_0^2)\in C(\Q)$, so there exists $v\in\Q$, $v\ne 1$, such that $(t_0,x_0^2)=\psi(v)$; in particular
\[x_0^2=\frac{16(2v^3 - 6v^2 + 6v - 1)(v^2 - 2v + 2)}{(v - 1)^4}.\]

Letting $w=x_0(v-1)^2/4$ we obtain
\[w^2=(2v^3 - 6v^2 + 6v - 1)(v^2 - 2v + 2).\]

This equation defines a hyperelliptic curve of genus 2 whose Jacobian has rank 1. A Chabauty computation\footnote{For curves of genus 2 whose Jacobian has rank 1, the set of rational points can be determined using the \texttt{Chabauty} function in \textsc{Magma}; see \S\ref{rational_points_section} for further details.} shows that the affine rational points on the curve are $(1,\pm 1)$ and $(3/2,\pm 5/4)$. It follows that $v=3/2$ and therefore $(t_0,x_0^2)=\psi(3/2)=(-5,400)$, so $(t_0,x_0)=(-5,\pm 20)$. Hence $\calB_4(\Q)=\{(-5,\pm 20)\}$.
\end{proof}

\begin{lem}\label{B7_curve_lem}
Let $f(t,x)=x^6 + 4tx^4 - 4x^3 - 4tx^2 - 1$ and let $\calB_7$ be the affine curve defined by the equation $f(t,x)=0$. Then
\[\calB_7(\Q)=\{(\rho(z),z):z\in\Q\setminus\{0,\pm 1\}\},\]

where $\rho$ is the rational function defined in \eqref{rational_fncn_def}.
\end{lem}
\begin{proof}
The equation $f(t,x)=0$ can be solved for $t$ to obtain $t=\rho(x)$.
\end{proof}

Having determined the rational points on all the necessary curves, we proceed to the proof of our main result.

\subsection{Classification of Galois groups}\label{group_classification_section}

In this section we apply the techniques of \S\ref{hit_section} to study the specializations of the polynomial $\Phi_4(t,x)$. We will use here much of the same notation as in \S\ref{hit_section}. Thus, $\Delta(t)\in\Q[t]$ is the discriminant of $\Phi_4$, $N/\Q(t)$ is a splitting field of $\Phi_4$, and $G=\Gal(N/\Q(t))$. For $c\in\Q$, $\Phi_{4,c}$ denotes the specialized polynomial $\Phi_4(c,x)$, and $G_c$, $\calF_c$ denote the Galois group and factorization type of $\Phi_{4,c}$. The ideal $(t-c)\subset\Q[t]$ is denoted by $\p_c$. If $\P$ is a prime of $N$, then $G_{\P}$ denotes the decomposition group of $\P$ over $\Q(t)$. Finally, the exceptional set of $\Phi_4$ is the set $\calE(\Phi_4)$ consisting of all rational numbers $c$ for which either $\Phi_{4,c}$ is reducible or $G_c$ is not isomorphic to $G$.

The rational functions $\lambda, \eta, \rho$ defined in \eqref{rational_fncn_def} will play a prominent role in what follows. Recall that $\Im\lambda$ denotes the set of all rational numbers of the form $\lambda(r)$ where $r\in\Q^{\ast}$, and similar notation applies to $\eta$ and $\rho$.

\begin{prop}\label{Phi4_exceptional_set_prop}
The exceptional set of $\Phi_4$ is given by
\[\calE(\Phi_4)=\Im\lambda\cup\Im\eta\cup\{-5/2\}.\]
\end{prop}
\begin{proof}
We apply Proposition \ref{HIT_exceptional_set_prop} to the polynomial $\Phi_4$. As seen in \S\S\ref{subgroup_lattice_section}-\ref{subgroup_curves_section}, the group $G$ has five maximal subgroups up to conjugacy, and their fixed fields are generated by the polynomials $f_1,\ldots, f_5$ defined before Lemma \ref{maximal_subgroup_curves_lem}. The only rational root of $\Delta(t)$ is $-5/4$, and the discriminants of the $f_i$ have no rational roots. Thus, for $c\in\Q\setminus\{-5/4\}$ we know that $c$ is in $\calE(\Phi_4)$ if and only if one of the polynomials $f_i(c,x)$ has a rational root, or equivalently, $c$ is the first coordinate of a rational point on one of the curves $f_i(t,x)=0$. Hence, Lemma \ref{maximal_subgroup_curves_lem} implies that
\[c\in\calE(\Phi_4)\iff c\in\Im\lambda\cup\Im\eta\cup\{0,-5/2\}.\]

This equivalence also holds for $c=-5/4$, since the polynomial $\Phi_4(-5/4, x)$ is reducible and $-5/4=\eta(2)\in\Im\eta$. To prove the proposition it remains only to observe that $0=\eta(1)\in\Im\eta$.
\end{proof}

\begin{lem}\label{lambda_eta_images_lem}
The intersection of $\Im\lambda$ and $\Im\eta$ is given by
\[\Im\lambda\cap\Im\eta=\{19/16\}.\]
\end{lem}
\begin{proof}
Note that $19/16=\lambda(8)=\eta(1/2)$, so $19/16\in\Im\lambda\cap\Im\eta$. Suppose now that $u,v\in\Q^{\ast}$ satisfy $\lambda(u)=\eta(v)$. Then we have
\[2uv^3+vu^2+8uv-8u-4v=0.\]

Letting $y=v^3 + uv + 4v - 4$, the above equation implies that
\[y^2=(v^2 + 4)(v^4 + 4v^2 - 8v + 4).\]

This equation defines a hyperelliptic curve of genus 2 whose Jacobian has rank 1. A Chabauty computation shows that the only affine rational points on the curve are $(0,\pm 4)$ and $(1/2,\pm17/8)$. It follows that $v=1/2$ and therefore $\eta(v)=19/16$.
\end{proof}

We can now begin to determine the structure of the groups $G_c$ and the factorization types $\calF_c$ for all $c\in\calE(\Phi_4)$.

\begin{prop}\label{lambda_image_prop}
Suppose that $c\in\Im(\lambda)\setminus\{\frac{19}{16}\}$.
Then $G_c\equiv M_1$ and $\calF_c=\{12\}$.
\end{prop}
\begin{proof}
Recall that $\lambda(z)=(z^2 + 2z - 4)/(8z)$. As seen in \S\ref{subgroup_curves_section}, the fixed field of $M_1$ has the form $\Q(t)(\alpha)$, where $\alpha$ is a root of the polynomial $f_1(t,x)= x^2 - 64t^2 + 32t - 20$. Letting $\theta=4t+\alpha/2-1$ we then have $\lambda(\theta)=t$ and $\Q(t)(\theta)=\Q(t)(\alpha)$.

We now apply Proposition \ref{rational_function_galois_group_prop} with $\mu=\lambda$ and $P=\Phi_4$. Note that the polynomial $f(t,x)=t(-8x)-(4-2x-x^2)$ is monic and irreducible and has $\theta$ as a root. In the notation of Proposition \ref{rational_function_galois_group_prop} we have $H=M_1$.

Since $c\in\Im(\lambda)$, we may write $c=\lambda(v)$ for some $v\in\Q^{\ast}$. It is easy to see that $c\ne-5/4$, so $\Delta(c)\ne 0$. Moreover, $\disc f(t,x)=(2 - 8t)^2 + 16$ has no rational root, so $\disc f(c,x)\ne 0$. Hence, by part (1) of Proposition \ref{rational_function_galois_group_prop} there exists a prime $\P$ of $N$ dividing $\p_c$ such that $G_{\P}\subseteq M_1$.

Letting $g(t,x)=(8t)^6\cdot\Phi_4(\lambda(t),x)$ we find that $g\in\Q[t][x]$ is irreducible of degree 12, and the only rational root of $\disc g$ is 0. Hence $\disc g(v,x)\ne 0$.

We claim that $G_{\P}=M_1$. Assuming this for the moment, part (3) of Proposition \ref{rational_function_galois_group_prop} implies that $G_c\equiv M_1$ and $\calF_c=\calF(g)=\{12\}$, which is what we wish to prove. Thus, it remains only to show that $G_{\P}=M_1$. 

In what follows we will replace $\P$ by a conjugate ideal whenever necessary. Suppose by contradiction that $G_{\P}$ is a proper subgroup of $M_1$. Then $G_{\P}$ is contained in one of the groups $A_1,\ldots, A_5$. We consider each case separately.

If $G_{\P}\subseteq A_1$, then by \eqref{conjugate_pairs} we have $G_{\P}\subseteq C_5$ and hence $G_{\P}\subseteq M_3$. We now apply Proposition \ref{HIT_subgroup_curve_prop} with $H=M_3$. Recall that the fixed field of $M_3$ is generated by a root of the polynomial $f_3(t,x)$ defined in  \S\ref{subgroup_curves_section}. The discriminant of $f_3$ has no rational root, so $\Delta(c)\cdot\disc f(c,x)\ne 0$. By Proposition \ref{HIT_subgroup_curve_prop}, the polynomial $f_3(c,x)$ must have a rational root. Thus \eqref{f3_curve_points} implies that $c=0$ or $c=-5/2$. However, this contradicts the hypotheses since neither 0 nor $-5/2$ are in $\Im\lambda$.

If $G_{\P}\subseteq A_2$, then by \eqref{conjugate_pairs} we have $G_{\P}\subseteq B_6$ and hence $G_{\P}\subseteq M_2$, so $f_2(c,x)$ has a rational root. Since also $G_{\P}\subseteq M_1$, then $f_1(c,x)$ has a rational root. Now \eqref{f1_curve_points} and \eqref{f2_curve_points} imply that $c\in\Im\lambda\cap\Im\eta$, so $c=19/16$ by Lemma \ref{lambda_eta_images_lem}. However, this again contradicts the hypotheses.

If $G_{\P}\subseteq A_3$, then by \eqref{conjugate_pairs} we have $G_{\P}\subseteq D_3$ and hence $G_{\P}\subseteq M_4$, so $f_4(c,x)$ has a rational root. However, this is impossible by \eqref{f4,5_curve_points}.

Finally, if $G_{\P}$ is contained in either $A_4$ or $A_5$, then Lemmas \ref{A4_curve_lem} and \ref{A5_curve_lem} imply that $c=-5/4$ or $c=1/4$. We have already observed that $c$ cannot be $-5/4$, and for $c=1/4$ one can verify directly that $G_c$ has order 192. By Lemma \ref{HIT_lem} this implies that $G_{\P}$ has order 192, so cannot be contained in $A_4$ or $A_5$. Thus we have a contradiction.

Since every case has led to a contradiction, we conclude that $G_{\P}=M_1$, as claimed.
\end{proof}

Next we determine the Galois groups $G_c$ and factorization types $\calF_c$ when $c\in\Im\eta$. Note that $\rho(z)=\eta((z^2-1)/z)$, so that $\Im\rho\subseteq\Im\eta$. As we will see, the analysis when $c\in\Im\eta$ differs depending on whether $c\in\Im\rho$ or not.

\begin{prop}\label{eta_image_prop}
Suppose that $c\in\Im(\eta)\setminus\Im\rho$ and $c\notin\left\{-5,-2,-\frac{5}{4},0,\frac{19}{16}\right\}$. Then $G_c\equiv M_2$ and $\calF_c=\{8,4\}$.
\end{prop}
\begin{proof}
Recall that $\eta(z)=(4-3z-z^3)/(4z)$. As seen in \S\ref{subgroup_curves_section}, the fixed field of $M_2$ has the form $\Q(t)(\alpha)$, where $\alpha$ is a root of the polynomial $f_2(t,x)= x^3 + (4t + 3)x + 4$. Letting $\theta=-\alpha$ we then have $\eta(\theta)=t$ and $\Q(t)(\theta)=\Q(t)(\alpha)$.

We now apply Proposition \ref{rational_function_galois_group_prop} with $\mu=\eta$ and $P=\Phi_4$. The polynomial $f(t,x)=t(4x)-(4-3x-x^3)$ is monic and irreducible and has $\theta$ as a root. Moreover, $\disc f$ has no rational root.

Since $c\in\Im(\eta)$ we may write $c=\eta(v)$ for some $v\in\Q^{\ast}$. Given that $c\ne-5/4$, we have $\Delta(c)\cdot\disc f(c,x)\ne 0$. Thus part (1) of Proposition \ref{rational_function_galois_group_prop} implies that there exists a prime $\P$ of $N$ dividing $\p_c$ such that $G_{\P}\subseteq M_2$.

 Letting $g(t,x)=(4t)^6\cdot\Phi_4(\eta(t),x)$ we find that $g$ factors as 
 \begin{equation}\label{8_4_factorization}
(4t)^6\cdot\Phi_4(\eta(t),x)=F(t,x)\cdot J(t,x),
\end{equation}

where $F(t,x)\in\Z[t][x]$ is irreducible of degree 8 and
\begin{multline*}
J(t,x) = 16t^2x^4 + 16t^3x^3 + (8t^3 - 24t^2 + 32t)x^2  -(4t^5 - 8t^4 + 20t^3- 32t^2)x \\ - t^6 + 2t^5 - 4t^4 + 6t^3 + 5t^2 - 8t + 16
\end{multline*}
is also irreducible. Hence $\calF(g)=\{8,4\}$. The only rational roots of $\disc g$ are 0 and 2, so $\disc g(v,x)\ne 0$ since $v=2$ would imply that $c=\eta(2)=-5/4$, a contradiction.

We claim that $G_{\P}=M_2$. Assuming this for the moment, part (3) of Proposition \ref{rational_function_galois_group_prop} implies that $G_c\equiv M_2$ and $\calF_c=\calF(g)=\{8,4\}$, which is what we wish to prove. Hence the proof will be complete if we show that $G_{\P}=M_2$. 

In what follows we will replace $\P$ by a conjugate ideal whenever necessary. Suppose by contradiction that $G_{\P}$ is a proper subgroup of $M_2$. Then $G_{\P}$ is contained in one of the groups $B_1,\ldots, B_7$.
 
If $G_{\P}\subseteq B_1$, then by \eqref{conjugate_pairs} we have $G_{\P}\subseteq D_2$ and hence $G_{\P}\subseteq M_4$. By Proposition \ref{HIT_subgroup_curve_prop} applied to $H=M_4$, this implies that the polynomial $f_4(c,x)$ has a rational root. However, this is impossible by \eqref{f4,5_curve_points}.
 
If $G_{\P}$ were contained in $B_2$, $B_3$, or $B_4$, then Lemmas \ref{B2_curve_lem}-\ref{B4_curve_lem} would imply that $c\in\{-5,-2,0\}$, which contradicts the hypotheses.

If $G_{\P}\subseteq B_5$, then by \eqref{conjugate_pairs} we have $G_{\P}\subseteq E_2$ and hence $G_{\P}\subseteq M_5$, so $f_5(c,x)$ has a rational root. However, this is impossible by \eqref{f4,5_curve_points}.

If $G_{\P}\subseteq B_6$, then by \eqref{conjugate_pairs} we have $G_{\P}\subseteq A_2$ and hence $G_{\P}\subseteq M_1$, so $f_1(c,x)$ has a rational root. Since also $G_{\P}\subseteq M_2$, the polynomial $f_2(c,x)$ has a rational root, so \eqref{f1_curve_points} and \eqref{f2_curve_points} imply that $c\in\Im\lambda\cap\Im\eta$. Thus $c=19/16$ by Lemma \ref{lambda_eta_images_lem}. However, this contradicts the hypotheses.

Finally, if $G_{\P}\subseteq B_7$, then Lemma \ref{B7_curve_lem} implies that $c\in\Im\rho$, a contradiction.

Since every case has led to a contradiction, we conclude that $G_{\P}=M_2$, as claimed.
\end{proof}

To complete our analysis of the groups $G_c$ and factorization types $\calF_c$ it remains to consider the case where $c\in\Im\rho$. In this case the group $G_c$ will be isomorphic to a subgroup of $B_7$, and we claim that in fact $G_c\equiv B_7$. To prove this we could proceed by contradiction as above and consider the curves corresponding to the maximal subgroups of $B_7$. Unfortunately, the problem of finding all the rational points on these curves becomes unmanageable, so we will take a slightly different approach. We begin by proving two auxiliary results.

\begin{lem}\label{degree_8_specialization_lem}
Let $F\in\Q[t,x]$ be defined by the factorization \eqref{8_4_factorization}. Then for all $v\in\Q\setminus\{0,\pm 1\}$ the polynomial $F_v=F(v,x)$ is irreducible and its Galois group satisfies $\Gal(F_v)\cong\Gal(F)\cong(32,11)$.
\end{lem}
\begin{proof}
Factoring $F$ and computing its Galois group we find that $F$ is irreducible and $\Gal(F)\cong (32,11)$. The group $\Gal(F)$ has exactly three maximal subgroups, all of index 2. The fixed fields of these subgroups are generated by the polynomials
\begin{align*}
p_1(t,x)&= x^2 + 2^8t^5(x + 1024t^2),\\
p_2(t,x) &= x^2 - 2^{34}t^{14}(t^4 + 4t^2 - 8t + 4),\\
p_3(t,x) &= x^2 - 2^{50}t^{21}(t^3 - 16)(t^4 + 4t^2 - 8t + 4).\\
\end{align*}

\vspace{-5mm}
The discriminants of the $p_i$ are given by 
\begin{align*}
\disc p_1(t,x) &= 2^{16}t^7(t^3 - 16),\\
\disc p_2(t,x) &= 2^{36}t^{14}(t^4 + 4t^2 - 8t + 4),\\
\disc p_3(t,x) &= 2^{52}t^{21}(t^3-16)(t^4 + 4t^2 - 8t + 4).\\
\end{align*}

\vspace{-5mm}
We will now use Proposition \ref{HIT_exceptional_set_prop} to determine the exceptional set of $F$. A simple calculation shows that if $v\in\Q^{\ast}$, then $v$ is not a root of the discriminant of $F$ nor of its leading coefficient, and that $\disc p_i(v,x)\ne0$ for $i=1,2,3$. Hence Proposition \ref{HIT_exceptional_set_prop} implies that $v\in\calE(F)$ if and only if one of the polynomials $p_i(v,x)$ has a rational root.

If $p_1(v,x)$ has a rational root, then $\disc p_1(v,x)$ is a square, so there exists $y\in\Q$ such that $y^2=v(v^3-16)$. The hyperelliptic curve $X_1$ defined by this equation is isomorphic to the elliptic curve with Cremona label 108a1, which has exactly three rational points. Thus $X_1$ has three rational points, which must be the two points at infinity and the point $(0,0)$. Hence $v=0$, which is a contradiction. 

If $p_2(v,x)$ has a rational root, then there exists $y\in\Q$ such that
\[y^2=v^4 + 4v^2 - 8v + 4.\]

The hyperelliptic curve $X_2$ defined by this equation is isomorphic to the elliptic curve with Cremona label 20a2, which has exactly six rational points. Thus $X_2$ has six rational points, which must be the two points at infinity and the points $(0,\pm 2)$ and $(1,\pm 1)$. Since $v\ne 0$, this implies $v=1$.

If $p_3(v,x)$ has a rational root, then there exists $y\in\Q$ such that
\[y^2=v(v^3-16)(v^4 + 4v^2 - 8v + 4).\]

The hyperelliptic curve $X_3$ defined by this equation is isomorphic to the curve $C$ from Proposition \ref{genus3_chabauty}. Since $C$ has exactly five rational points, the same holds for $X_3$. A search for points on $X_3$ then shows that its only affine rational points are $(0,0)$ and $(-1,\pm 17)$. Hence $v=-1$.

We conclude from the above discussion that if $v\in\Q\setminus\{0,\pm 1\}$, then $v\notin\calE(F)$, and therefore $F_v$ is irreducible and $\Gal(F_v)\cong\Gal(F)$.
\end{proof}

\begin{prop}\label{rho_image_prop}
Suppose that $c\in\Im\rho\setminus\{-\frac{155}{72}\}$. Then $G_c\equiv B_7$ and $\calF_c=\{8,2,2\}$.
\end{prop}
\begin{proof}
Recall that
\[\rho(z)=\frac{1+4z^3-z^6}{4z^2(z^2-1)}=\eta\left(\frac{z^2-1}{z}\right).\]

The fixed field of $B_7$ has the form $\Q(t)(\theta)$, where $\theta$ is a root of the polynomial $f(t,x)=x^6 + 4tx^4 - 4x^3 - 4tx^2 - 1$. Thus $\theta\in N$ and $\rho(\theta)=t$.

We now apply Proposition \ref{rational_function_galois_group_prop} with $\mu=\rho$ and $P=\Phi_4$. Note that \[f(t,x)=t(4x^2(x^2-1))-(1+4x^3-x^6)\]

is monic and irreducible and has $\theta$ as a root. Moreover, the discriminant of $f$ has no rational root. Since $c\in\Im(\rho)$, we may write $c=\rho(s)$ for some $s\in\Q\setminus\{0,\pm 1\}$. By part (1) of Proposition \ref{rational_function_galois_group_prop} there exists a prime $\P$ of $N$ dividing $\p_c$ such that $G_{\P}\subseteq B_7$.

Let $\nu(t)=(t^2-1)/t$, so that $\rho(t)=\eta(\nu(t))$, and let
\[g(t,x)=\left(4t^2(t^2-1)\right)^6\cdot\Phi_4(\rho(t),x).\]

From the factorization \eqref{8_4_factorization} it follows that $g=A\cdot B$, where
\[A(t,x)=t^{12}\cdot F(\nu(t),x)\;\;\text{and}\;\; B(t,x)=t^6\cdot J(\nu(t),x).\]

One can verify that $A,B\in\Q[t][x]$ and that $A$ is irreducible of degree 8. Factoring $B$ we find that $B=P\cdot Q$, where
\begin{align*}
P(t,x) &= (4t^4 - 4t^2)x^2 + (4t-4t^3)x - t^6 + 2t^5 + 2t^3 + 2t^2 - 1,\\
Q(t,x) &= (4t^4 - 4t^2)x^2 + (4t^5 - 4t^3)x + t^6 - 2t^4 + 2t^3 + 2t + 1,
\end{align*}

and $P$, $Q$ are irreducible. Note that $\calF(g)=\{8,2,2\}$ since $g=A\cdot P\cdot Q$. Computing the discriminant of $g$ we see that its only rational roots are $0,\pm 1$. Hence $\disc g(s,x)\ne 0$.

We claim that $G_{\P}=B_7$. Assuming this for the moment, part (3) of Proposition \ref{rational_function_galois_group_prop} implies that $G_c\equiv B_7$ and $\calF_c=\calF(g)=\{8,2,2\}$, which is what we have to prove. Thus it remains only to show that $G_{\P}=B_7$.

Since $G_{\P}$ is a subgroup of $B_7$ and $|B_7|=64$, it suffices to show that $|G_{\P}|=64$. We know by Lemma \ref{HIT_lem} that $G_{\P}$ is isomorphic to $G_c$, so we may show instead that the splitting field of $\Phi_{4,c}$ has degree 64 over $\Q$. Let $E$ denote this splitting field.

Let $v=\nu(s)=(s^2-1)/s$, so that $c=\rho(s)=\eta(v)$. From our work above we deduce that, up to a nonzero scalar,
\[\Phi_{4,c}(x)=F(v,x)\cdot P(s,x)\cdot Q(s,x).\]

Hence $E$ is the compositum of the splitting fields of $F(v,x)$ and $P(s,x)\cdot Q(s,x)$. Note that $P(s,x)$ and $Q(s,x)$ split over the same field, namely
\[K=\Q\left(\sqrt{(s^4-1)(s^2-2s-1)}\right).\]

This field is equal to $\Q$ if and only if there exists $y\in\Q$ such that
\[y^2=(s^4-1)(s^2-2s-1).\]

The hyperelliptic curve defined by this equation has a Jacobian of rank 0; computing the points on the curve we find that its only affine rational points are $(0,\pm 1)$ and $(\pm1,0)$. Since $s\not\in\{0,\pm1\}$, we conclude that $K$ is a quadratic field.

Let $L$ be the splitting field of $F(v,x)$. By Lemma \ref{degree_8_specialization_lem} we know that $[L:\Q]=32$. Since $E=LK$, we will have $[E:\Q]=64$ unless $K\subset L$. We will now rule out this possibility and thus complete the proof.

As seen in the proof of Lemma \ref{degree_8_specialization_lem}, the group $\Gal(F)$ has exactly three subgroups of index 2; by the lemma the same is true for $\Gal(F(v,x))$. Hence the field $L$ has exactly three subfields of degree 2 over $\Q$. Now, the quadratic extensions of $\Q(t)$ that lie inside the splitting field of $F$ are generated by the polynomials $p_1, p_2, p_3$ listed in the proof of Lemma \ref{degree_8_specialization_lem}. Specializing at $v$ we conclude that the polynomials $p_i(v,x)$ all have a root in $L$. Hence $L$ contains the fields $L_i=\Q(\sqrt d_i)$, where
\[d_1=v(v^3-16),\;d_2=v^4+4v^2-8v+4,\;\text{and}\;d_3=d_1d_2.\]

From our work in the proof of Lemma \ref{degree_8_specialization_lem} we deduce that the $L_i$ are quadratic fields, and all distinct. Hence these must be the three quadratic fields contained in $L$. We claim that none of these fields can be equal to $K$. 

Using the relation $v=(s^2-1)/s$ we see that $L_i=\Q(\sqrt{\delta_i})$, where
\[\delta_1=(s^2 - 1)(s^6 - 3s^4 - 16s^3 + 3s^2 - 1),\;\;\delta_2=s^8 - 8s^5 + 2s^4 + 8s^3 + 1,\]

and $\delta_3=\delta_1\delta_2$. Suppose that $K=L_i$ for some $i$. Then $\delta_i\cdot (s^4-1)(s^2-2s-1)$ is a rational square. It follows that there exists $y\in\Q$ such that one of the following holds:
\begin{align*}
y^2 &= (s^2 + 1)(s^2 - 2s - 1)(s^6 - 3s^4 - 16s^3 + 3s^2 - 1),\\
y^2 &= (s^4 - 1)(s^2 - 2s - 1)(s^8 - 8s^5 + 2s^4 + 8s^3 + 1),\\
y^2 &= (s^2 + 1)(s^2 - 2s - 1)(s^6 - 3s^4 - 16s^3 + 3s^2 - 1)(s^8 - 8s^5 + 2s^4 + 8s^3 + 1).
\end{align*}

By Propositions \ref{genus4_chabauty}, \ref{genus6_descent}, and \ref{genus8_descent_chabauty}, this implies that $s\in\{3,-1/3\}$. Then $c= \rho(s)=-155/72$, which contradicts the hypotheses. This proves that $K$ cannot be contained in $L$, as claimed.
\end{proof}

Summarizing our work up to this point we obtain a proof of Theorem \ref{phi4_factorization_galois_intro_thm}.

\begin{thm}\label{phi4_factorization_galois_thm}
Let $c\in\Q\setminus \left\{ -5, -\frac{5}{2}, -\frac{155}{72}, -2, -\frac{5}{4}, 0, \frac{19}{16}\right\}$. Then the following hold:
\begin{enumerate}
\item If $c\notin\Im(\lambda)$ and $c\notin\Im(\eta)$, then $G_c\equiv\calW$ and $\calF_c=\{12\}$.
\item If $c\in\Im(\lambda)$, then $G_c\equiv\calG$ and $\calF_c=\{12\}$.
\item Suppose that $c\in\Im(\eta)$. Then the following hold:
\begin{enumerate}
\item If $c\not\in\Im(\rho)$, then $G_c\equiv\calH$ and $\calF_c=\{8,4\}$.
\item If $c\in\Im(\rho)$, then $G_c\equiv\calI$ and $\calF_c=\{8,2,2\}$.
\end{enumerate}
\end{enumerate}
\end{thm}
\begin{proof}
Suppose that $c\notin\Im(\lambda)$ and $c\notin\Im(\eta)$. Then Proposition \ref{Phi4_exceptional_set_prop} implies that $c$ is not in the exceptional set of $\Phi_4$; thus $G_c\cong G$ and $\calF_c=\calF(\Phi_4)=\{12\}$. It follows from Lemma \ref{HIT_lem} that $G_c\equiv G$ and hence $G_c\equiv\calW$ by \eqref{subgroup_equivalence}. This proves (1). The remaining statements follow immediately from \eqref{subgroup_equivalence} and Propositions \ref{lambda_image_prop}, \ref{eta_image_prop}, and \ref{rho_image_prop}.
\end{proof}

For the numbers $c$ that are excluded in Theorem \ref{phi4_factorization_galois_thm}, the Galois group $G_c$ can be computed using an algorithm of Fieker and Kl\"uners \cite{fieker-kluners} that is implemented in \textsc{Magma}. Carrying out these Galois group computations and factoring $\Phi_{4,c}$ to find the factorization types $\calF_c$, we obtain the data summarized in Table \ref{excluded_gc_table} below. For each number $c$ we give a pair $(o,n)$ such that $G_c$ is isomorphic to the group labeled $(o,n)$ in the \textsc{Small Groups} library \cite{small_groups}; thus $o$ represents the order of $G_c$.

\medskip
\begin{table}[h!]
\centering
\begin{tabular}{ |c|c|c|c|c|c|c|c|} 
 \hline
$c$ & -5 & -5/2 & -155/72 & -2 & -5/4 & 0 & 19/16 \\
 \hline
 $\calF_c$ & $\{8,4\}$ & $\{12\}$ & $\{8,2,2\}$ & $\{8,4\}$ & $\{8,2,2\}$ & $\{8,4\}$ & $\{8,4\}$ \\
 \hline
 $G_c$ & $(64,20)$ & $(24,5)$ & $(32,11)$ & $(32,3)$ & $(64,101)$ & $(8,2)$ & $(64,101)$\\
 \hline
\end{tabular}
\bigskip
\caption{Factorization types $\calF_c$ and isomorphism types of Galois groups $G_c$ for excluded values of $c$.} 
\label{excluded_gc_table}
\end{table}

\subsection{Density results}\label{phi4_density_section}

For every polynomial $F\in\Q[x]$, let $S_{F}$ denote the set of all prime numbers $p$ such that $F$ has a root in $\Q_p$. The Chebotarev Density Theorem implies that the Dirichlet density of $S_F$, which we denote by $\delta(S_F)$, exists and can be computed if the Galois group of $F$ is known. More precisely, we have the following lemma.

\begin{lem}\label{galois_density_lem}
Let $F\in \Q[x]$ be a separable polynomial of degree $n\ge 1$. Let $L$ be a splitting field for $F$, and set $G=\Gal(L/\Q)$. Let $\alpha_1,\ldots, \alpha_n$ be the roots of $F$ in $L$ and, for each index $i$, let $G_i$ denote the stabilizer of $\alpha_i$ under the action of $G$. Then the Dirichlet density of $S_F$ is given by
\begin{equation}\label{density_formula}
\delta(S_F)=\frac{\left|\bigcup_{i=1}^n G_i\right|}{|G|}.
\end{equation}
\end{lem}
\begin{proof}
This follows from Theorem 2.1 in \cite{krumm_lgp}.
\end{proof}

We deduce from Lemma \ref{galois_density_lem} that the value of $\delta(S_F)$ can be determined if a permutation representation of $G$ (induced by a labeling of the roots of $F$) is known. Indeed, if $H$ is a subgroup of the symmetric group $S_n$ such that $G\equiv H$, then $\delta(S_F)$ can be computed using \eqref{density_formula} by replacing $G$ with $H$ and $G_i$ with $H_i$, the stabilizer of $i$ in $H$. This fact reduces the problem of computing $\delta(S_F)$ to the problem of computing a permutation representation of $G$. The latter can be done using the algorithm of Fieker and Kl\"uners referenced above. Hence it is in principle possible to evaluate $\delta(S_F)$ for any polynomial $F\in\Q[x]$.

Returning now to dynatomic polynomials, Theorem \ref{phi4_factorization_galois_thm} provides a permutation representation of the Galois group of $\Phi_{4,c}$ for all but seven rational numbers $c$. Thus, we may use this theorem together with Lemma \ref{galois_density_lem} to compute the density of the set $S_{\Phi_{4,c}}$ for almost all $c$. To simplify notation we will henceforth write $S_c$ instead of $S_{\Phi_{4,c}}$.

\begin{thm}\label{phi4_density_thm}
Let $c\in\Q\setminus \left\{ -5, -\frac{5}{2}, -\frac{155}{72}, -2, -\frac{5}{4}, 0, \frac{19}{16}\right\}$.
\begin{enumerate}
\item If $c\notin\Im(\lambda)$ and $c\notin\Im(\eta)$, then $\delta(S_c)=85/384$.
\item If $c\in\Im(\lambda)$, then $\delta(S_c)=43/192$.
\item Suppose that $c\in\Im(\eta)$. Then the following hold:
\begin{enumerate}
\item If $c\not\in\Im(\rho)$, then $\delta(S_c)=53/128$.
\item If $c\in\Im(\rho)$, then $\delta(S_c)=39/64$.
\end{enumerate}
\end{enumerate}
\end{thm}
\begin{proof}
In case (1) we know that $G_c\equiv\calW$. Using the description of $\calW$ given in Appendix \ref{galois_data_appendix}, we compute the stabilizers $\calW_i$ for $1\le i\le 12$ and find that the set $\bigcup_{i=1}^{12}\calW_i$ has cardinality 85. By Lemma \ref{galois_density_lem} this implies that
\[\delta(S_c)=85/|\calW|=85/384.\]
The remaining statements are proved in a similar way.
\end{proof}

For the numbers $c$ that are excluded in Theorem \ref{phi4_density_thm}, the density of $S_c$ can be found by first computing the Galois group of $\Phi_{4,c}$ and then applying Lemma \ref{galois_density_lem}. Carrying out these computations we obtain the density values shown in Table \ref{density_table}.

\medskip
\begin{table}[h!]
\centering
\begin{tabular}{ |c|c|c|c|c|c|c|c|} 
 \hline
$c$ & -5 & -5/2 & -155/72 & -2 & -5/4 & 0 & 19/16 \\
 \hline
 $\delta(S_c)$ & 23/64 & 1/6 & 1/2 & 11/32 & 39/64 & 1/4 & 27/64 \\
 \hline
\end{tabular}
\bigskip
\caption{Density of $S_c$ for excluded values of $c$.}
\label{density_table}
\end{table}
We end this section with an application of the density results obtained above. By a well-known theorem of Morton \cite{morton_period4}, a quadratic polynomial over $\Q$ cannot have a rational point of period four. We can now prove Theorem \ref{period_percentage_thm}, thus showing that there is a strong local obstruction explaining Morton's result.

\begin{cor}\label{period4_local_obstruction}
Let $f\in\Q[x]$ be a quadratic polynomial. Then, for more than $39\%$ of all primes $p$, $f$ does not have a point of period four in $\Q_p$.
\end{cor}
\begin{proof}
Let $c$ be the unique rational number such that $f$ is linearly conjugate to $\phi_c(x)=x^2+c$. The polynomials $\Phi_{4,f}$ and $\Phi_{4,c}$ factor in the same way over every extension of $\Q$ (see \cite[Cor. 2.11]{krumm_lgp}); in particular $\Phi_{4,f}$ has a root in a $p$-adic field $\Q_p$ if and only if $\Phi_{4,c}$ has a root in $\Q_p$. From Theorem \ref{phi4_density_thm} and Table \ref{density_table} we deduce that
\[1-\delta(S_c)\ge1-\frac{39}{64}=\frac{25}{64}>0.39.\]

Thus, more than $39\%$ of all primes are in the complement of $S_c$. Now, if $p\notin S_c$, then $\Phi_{4,c}$ does not have a root in $\Q_p$, so $\Phi_{4,f}$ does not have a root in $\Q_p$, and therefore $f$ does not have a point of period four in $\Q_p$.
\end{proof}

\section{Rational points on hyperelliptic curves}\label{rational_points_section}

In this section we describe the methods used to determine the rational points on the various hyperelliptic curves that arise in the proof of Theorem \ref{phi4_factorization_galois_thm}. In addition we provide detailed arguments for five of these curves.

\subsection{Preliminaries} Let $C$ be a hyperelliptic curve over $\Q$ of genus $g\ge 2$. By Faltings's Theorem the set $C(\Q)$ is finite. The question we are concerned with here is how to provably determine all the rational points on $C$. The techniques used for this differ depending on whether $C(\Q)$ is empty or not.

Suppose first that $C(\Q)=\emptyset$ and we need to prove this. An initial step is to test whether $C$ has a point over every completion of $\Q$. This can be done algorithmically, and the necessary algorithms have been implemented; see \cite[\S 5]{bruin_local_solvability}, for instance. The \textsc{Magma} functions \texttt{HasPoint} and \texttt{HasPointsEverywhereLocally} can be used to determine whether $C$ has points over a given $p$-adic field or over all $p$-adic fields. These methods are used in the proofs of Propositions \ref{genus4_curve}, \ref{genus6_descent}, and \ref{genus8_descent_chabauty} below.

If we find that $C$ fails to have a point over $\Q_p$ for some $p$, then $C(\Q)=\emptyset$. However, it may occur that $C$ has a point in every $p$-adic field and yet has no rational point. In that case we can apply the two-cover descent \cite{bruin-stoll_2cover_descent} or Mordell-Weil sieve \cite{bruin-stoll_mw_sieve} techniques to try to show that $C(\Q)$ is empty.  An implementation of the former method is available via the \textsc{Magma} function \texttt{TwoCoverDescent}. The algorithm determines a collection of covers of $C$ such that any rational point on $C$ lifts to one of the covers. In particular, if the algorithm has empty output, then $C(\Q)$ must be empty. We have applied this method in the proofs of Propositions \ref{genus4_curve} and \ref{genus8_descent_chabauty}.

Suppose now that $C$ does have a rational point, and we need to determine all rational points. It is useful for this purpose to know the structure of the Mordell-Weil group $J(\Q)$, where $J$ is the Jacobian variety of $C$. In particular, the rank $r$ of this group should be determined if possible. An upper bound for $r$ can be obtained by using a descent argument as in \cite{stoll_descent}, and a lower bound by searching for independent rational points on $J$. If $r=0$, then it is a finite calculation to determine all rational points on $C$. If $r>0$ the problem becomes much harder, though still possible to handle in certain cases. Most importantly for our purposes, if $r<g$, then one may be able to apply Chabauty techniques (described below) to obtain an upper bound for the number of rational points on $C$. If this upper bound matches the number of known points, then $C(\Q)$ has been determined. If the upper bound obtained is larger than the known number of rational points, then a Mordell-Weil sieve argument might improve the bound.

\textsc{Magma} includes several functions that can be used to carry out the above calculations. The \texttt{RankBound} function computes an upper bound for $r$, and if $r$ is known to be 0, then the \texttt{Chabauty0} function determines all rational points on $C$. These tools are used in the proofs of Lemmas \ref{maximal_subgroup_curves_lem} and \ref{A5_curve_lem}, and Propositions \ref{rho_image_prop} and \ref{genus6_descent}. If $r=1$ and $g=2$, then the \texttt{Chabauty} function uses a combination of the Chabauty-Coleman method with a Mordell-Weil sieve to compute the set $C(\Q)$. This method is used in the proofs of Lemmas \ref{B4_curve_lem} and \ref{lambda_eta_images_lem}, and Proposition \ref{genus8_descent_chabauty}.

It remains to discuss the Chabauty-Coleman techniques used in the proofs of Propositions \ref{genus3_chabauty} and \ref{genus4_chabauty}. We will now give a brief overview of the method; further details can be found in \cite{mccallum-poonen} and \cite{wetherell_thesis}.

\subsection{The method of Chabauty and Coleman}\label{chabauty_method_section} 

Let $C$ be a smooth, projective, geometrically integral curve over $\Q$ of genus $g\ge 2$, and suppose that $C$ has a rational point.  The goal of the method is to find an upper bound for $\#C(\Q)$. We begin by considering the reduction of $C$ modulo a prime.

Let $p$ be a prime of good reduction for $C$, and let $\calC$ be the minimal proper regular model\footnote{See \cite[\S10.1.1]{liu} or \cite[Chap. IV, \S4]{silverman_topics} for background on models of curves.} of $C$ over $\Z_p$. We denote the set of $\F_p$-points on the special fiber of $\calC$ by $C(\F_p)$. The model $\calC$ gives rise to a surjective reduction map\footnote{See  \cite[\S10.1.3]{liu} for the definition and properties of the reduction map.} $C(\Q_p)\onto C(\F_p)$ which we denote by $Q\mapsto\tilde Q$. The preimage of a point under this map will be called a \textit{residue disk}. Clearly the residue disks form a finite partition of $C(\Q_p)$, so in order to bound $\#C(\Q)$ it suffices to bound the number of rational points within each residue disk. To accomplish the latter we use integration of global differentials on $C$.

Let $J$ be the Jacobian variety of $C$. Recall that the $\Q_p$-vector space of global 1-forms on $J$ has dimension $g$, and that there is a canonical isomorphism
\begin{equation}\label{jacobian_1forms_isomorphism}
H^0(J,\Omega_{J/\Q_p}^1)\cong H^0(C,\Omega_{C/\Q_p}^1).
\end{equation}

To every global 1-form $\omega$ on $J$ there corresponds a unique analytic homomorphism $\lambda_{\omega}: J(\Q_p)\to\Q_p$ such that $d(\lambda_{\omega})=\omega$; see \cite[Lemma 3.2]{wetherell_thesis}. Using this map we define, for points $a,b\in J(\Q_p)$,
\[\int_a^b\omega=\lambda_{\omega}(b)-\lambda_{\omega}(a).\]

The integral thus defined satisfies the usual properties one might expect. In particular, it is linear in the integrand and changes sign if the limits of integration are interchanged. See Proposition 2.4 in \cite{coleman_integrals} for these and other properties.

A central object in the method of Chabauty and Coleman is the \textit{annihilator} of $J(\Q)$, which is defined to be the $\Q_p$-vector space
\[\Ann J(\Q)=\{\omega\in H^0(J,\Omega^1_{J/\Q_p}):\lambda_{\omega}(a)=0\;\;\text{for all}\;\;a\in J(\Q)\}.\]

An important step in the method is to find a nontrivial 1-form in this space. The existence of such a form is guaranteed if $r<g$, where $r$ is the rank of the group $J(\Q)$. Indeed, this can be deduced from the following lemma, which we include for later reference. The proof is a simple exercise in linear algebra.

\begin{lem}\label{annihilator_lem}
Let $\omega_0,\ldots,\omega_{g-1}$ form a basis for $H^0(J,\Omega^1_{J/\Q_p})$. Suppose that $D_1,\ldots, D_r\in J(\Q)$ generate a subgroup of finite index in $J(\Q)$. Let $M$ be the $r\times g$ matrix with entries $M_{ij}=\lambda_{\omega_j}(D_i)$. Then the linear isomorphism $\Q_p^g\to H^0(J,\Omega^1_{J/\Q_p})$ given by $(a_0,\ldots, a_{g-1})\mapsto a_0\omega_0+\cdots+a_{g-1}\omega_{g-1}$ restricts to an isomorphism $\ker M\to\Ann J(\Q)$.
\end{lem}

Now let $\eta$ be a global 1-form on $C$, and let $\omega$ be the 1-form on $J$ corresponding to $\eta$ under the isomorphism \eqref{jacobian_1forms_isomorphism}. For any points $P,Q\in C(\Q_p)$ we set
\[\int_P^Q\eta=\lambda_{\omega}([Q-P])=\int_0^{[Q-P]}\omega.\]

It follows from the definitions that if $\omega\in\Ann J(\Q)$, then
\begin{equation}\label{annihilating_integral_ppty}
\int_P^Q\eta=0\hspace{4mm}\text{for all}\;\;P,Q\in C(\Q).
\end{equation}

This property of annihilating differentials will allow us to obtain an upper bound for $\#C(\Q)$. Recall that the general strategy is to find an upper bound for the number of rational points in every residue disk. We can now describe precisely how to do this.

Suppose that $r<g$, so that $\Ann J(\Q)$ is nontrivial, and let $\calR\subseteq C(\Q_p)$ be a residue disk. In order to find an upper bound for the cardinality of the set $C(\Q)\cap\calR$ we begin by choosing a basepoint $P\in C(\Q)$ and a global differential $\eta$ on $C$ corresponding to a 1-form in $\Ann J(\Q)$. We impose no restriction on $P$, but for reasons mentioned below we require that
\begin{equation}\label{differential_reduction_requirement}
\text{$\eta$ reduces to a nonzero global differential on $C/\F_p$.}
\end{equation}

Let $\lambda:\calR\to\Q_p$ be the function defined by the formula
\[\lambda(Q)=\int_P^Q\eta.\]

By \eqref{annihilating_integral_ppty} this function vanishes on $C(\Q)\cap\calR$. It therefore suffices to find a bound for the number of zeros of $\lambda$. To do this we choose a point $Q\in\calR$ and a local parameter $t$ at $Q$, i.e.,  a generator of the maximal ideal of the local ring $\calO_{C,Q}$. We require that
\begin{equation}\label{local_parameter_requirement}
\text{$t$ reduces to a local parameter at the point $\tilde Q\in C(\F_p)$.}
\end{equation}

This assumption ensures that $t$ is regular on $\calR$ and induces a bijection $t:\calR\to p\Z_p$. We may thus regard $\lambda$ as a function $p\Z_p\to\Q_p$. The next step in the method is to find a power series representation for this function.

By embedding the local ring $\calO_{C,Q}$ in its completion we can represent every rational function in $\calO_{C,Q}$ as a formal power series in $\Q_p\llbracket t\rrbracket$, and every rational function in $\Q_p(C)$ as a formal Laurent series in $t$ with $\Q_p$-coefficients\footnote{More precisely, there is a $\Q_p$-algebra isomorphism $\Q_p\llbracket z\rrbracket\longrightarrow\widehat\calO_{C,Q}$ mapping $z$ to $t$. This induces an isomorphism of $\Q_p(\!(z)\!)$ with the field of fractions of $\widehat\calO_{C,Q}$, into which $\Q_p(C)$ naturally embeds. See \cite[p. 128]{bump} for proofs of these statements.}.

We may write $\eta=wdt$ for a unique rational function $w\in\calO_{C,Q}$. The condition \eqref{differential_reduction_requirement} implies that the power series representation of $w$ has $\Z_p$-coefficients, and that these are not all zero modulo $p$. Thus
\[w(t)=\sum_{i=0}^{\infty}c_it^i\]

with $c_i\in\Z_p$ for all $i$, and $\ord_p(c_i)=0$ for some $i$. Formal antidifferentiation now yields a power series representation for $\lambda$:
\begin{equation}\label{lambda_power_series}
\lambda(t)=b_0+\sum_{i=0}^{\infty}\frac{c_i}{i+1}t^{i+1},
\end{equation}
 
 where the constant term is necessarily given by $b_0=\int_P^Q\eta$. Recall that our goal is to bound the number of zeros of $\lambda$ in $p\Z_p$. It is helpful for this purpose to define an auxiliary function $F:\Z_p\to\Q_p$ by $F(x)=\lambda(px)$. From \eqref{lambda_power_series} we obtain the power series representation

\[F(t)=\sum_{i=0}^{\infty}b_it^i,\;\; \text{where}\;\; b_i=\frac{c_{i-1}}{i}\cdot p^i\;\;\text{for}\;\; i\ge 1.\]

We can apply Stra{\ss}mann's Theorem \cite[Thm. 4.4.6]{gouvea} to find an upper bound for the number of zeros of $F$ (and hence the number of zeros of $\lambda$). Defining an index $N$ uniquely by the conditions
\[\ord_p(b_N)=\min_{i\ge 0}\ord_p(b_i)\;\;\;\text{and}\;\;\;\ord_p(b_N)<\ord_p(b_i)\;\;\text{for}\;\;i>N,\]

the theorem guarantees that $F$ has at most $N$ zeros in $\Z_p$, and therefore
\[\#\left(C(\Q)\cap\calR\right)\le N.\]

This accomplishes our goal of finding a bound for the number of rational points in $\calR$. By applying this process to every residue disk we then obtain an upper bound for $\#C(\Q)$.

\begin{rem}\label{strassmann_bound_rem} In order to determine the value of $N$ one only needs to know finitely many of the coefficients $b_i$. Indeed, it is a simple exercise to show that if $n$ is large enough, then
\begin{equation}\label{power_series_precision}
n-\log_p(n)\ge\min\{\ord_p(b_i): 0\le i\le n\}.
\end{equation}

With $n$ as above, let $v$ be the quantity on the right-hand side of \eqref{power_series_precision}. Note that for all $i\ge 1$ we have
\[\ord_p(b_i)\ge i-\ord_p(i)\ge i-\log_p(i).\]

It follows from the choice of $n$ that $\ord_p(b_i)>v$ for all $i>n$. This implies that $v=\min_{i\ge 0}\ord_p(b_i)$. Moreover, since $\ord_p(b_N)=v$, we must have $N\le n$. Thus $N$ can be determined if the coefficients $b_0,\ldots, b_n$ are known.
\end{rem}

\subsection{Explicit Chabauty for hyperelliptic curves}\label{hyperelliptic_chabauty_section}
 
Though the technique described above can in theory be applied to a fairly general curve $C$, in practice there are several issues that make it difficult to carry out the necessary computations. For hyperelliptic curves, however, there are tools available which in many cases make it possible to explicitly compute all of the required data.

We assume henceforth that $C$ is a hyperelliptic curve given by an equation $y^2=f(x)$, where $f(x)$ has odd degree and the genus $g$ is at least 2. Let $p$ be a prime not dividing the discriminant of $f(x)$, so that $C$ has good reduction at $p$. Work of Balakrishnan-Bradshaw-Kedlaya \cite{balakrishnan-bradshaw-kedlaya} provides an algorithm for computing the integral $\int_P^Q\eta$ for any points $P,Q\in C(\Q_p)$ and any global differential $\eta$ on $C$; an implementation of the algorithm is included in \textsc{Sage} \cite{sage}. By using this tool it is possible to automate a significant portion of the calculations required for a Chabauty-Coleman argument.

To begin we assume that the rank $r$ of $J(\Q)$ has been determined, and that $r<g$. From our discussion in the previous section we deduce the following steps for bounding the number of rational points on $C$.

\begin{enumerate}
\item\label{annihilator_step} Determine a basis for the annihilator of $J(\Q)$.
\item Choose a basepoint $P\in C(\Q)$.
\item For every point $\tilde Q\in C(\F_p)$:
\begin{enumerate}
\item\label{lifting_step} Lift $\tilde Q$ to a point $Q\in C(\Q_p)$.
\item\label{parameter_step} Choose a local parameter $t$ at $Q$ satisfying \eqref{local_parameter_requirement}.
\item\label{differential_step} Choose a global differential $\eta\in\Ann J(\Q)$ satisfying \eqref{differential_reduction_requirement}.
\item\label{series_step} Find the power series $w(t)\in\Z_p\llbracket t\rrbracket$ such that $\eta=w(t)dt$.
\item Construct the power series
\[F(t)=\int_P^Q\eta\;+\;I(pt),\]
where $I(t)$ is the formal antiderivative of $w(t)$.
\item\label{strassmann_step} Compute the Stra{\ss}mann bound for $F(t)$.
\end{enumerate}
\end{enumerate}

By adding all the bounds computed in step \ref{strassmann_step} we obtain an upper bound for $\#C(\Q)$. We will now discuss how to explicitly carry out each of the above steps.

\smallskip
\noindent\underline{\textit{Step \ref{annihilator_step}.}} Since this will always be the case in the examples worked out in this paper, let us assume that we know $r$ pairs of points $P_i, Q_i\in C(\Q)$ such that the divisor classes $[P_i-Q_i]$ generate a subgroup of finite index in $J(\Q)$.

A basis for the global differentials on $C$ is given by
\[\omega_0=(1/2y)dx\;\;\text{and}\;\;\omega_i=x^i\omega_0\;\;\text{for}\;\;i=1,\ldots, g-1.\]

Let $M$ be the $r\times g$ matrix with entries
\[M_{ij}=\int_{P_i}^{Q_i}\omega_j.\]

Computing the entries of $M$ and a basis for $\ker M$, Lemma \ref{annihilator_lem} can be applied to obtain a basis for $\Ann J(\Q)$. If necessary, we then scale the basis differentials to ensure that they reduce to nonzero global differentials on $C/\F_p$. Let $\eta_1,\ldots, \eta_s$ be a basis for $\Ann J(\Q)$ constructed in this way. Note that any one of the $\eta_i$'s can be chosen as the differential $\eta$ in step \ref{differential_step}. More generally, $\eta$ can be any linear combination of the $\eta_i$'s satisfying \eqref{differential_reduction_requirement}.

\smallskip
\noindent\underline{\textit{Step \ref{lifting_step}.}} If $\tilde Q$ is the point at infinity, then we lift to $Q=\infty$. Otherwise we apply Hensel's Lemma. Let $\tilde Q$ have coordinates $(\tilde a,\tilde b)$. If $\tilde b\ne 0$ we lift $\tilde b$ to a square root of $f(\tilde a)$ in $\Z_p$. If $\tilde b=0$ we lift $\tilde a$ to a root of $f(x)$ in $\Z_p$.

\smallskip
\noindent\underline{\textit{Step \ref{parameter_step}.}} Assuming the lift $Q$ of $\tilde Q$ was constructed as above, we choose the local parameter $t$ at $Q$ as follows:
\[
t=
\begin{cases}
x-a & \text{if } Q=(a,b)\text{ with }b\ne 0,\\
y & \text{if } Q=(a,b)\text{ with }b=0, \\
y/x^{g+1} & \text{if } Q=\infty.
\end{cases}
\]

\smallskip
\noindent\underline{\textit{Step \ref{series_step}.}} Recall that $\eta$ was constructed as a linear combination of the differentials $\eta_i$ from Step \ref{annihilator_step}, say $\eta=c_1\eta_1+\cdots+c_s\eta_s$. Expressing each $\eta_i$ as a linear combination of $\omega_0,\ldots,\omega_{g-1}$ we can write $\eta_i=p_i(x)\omega_0$ for some polynomial $p_i(x)$. Thus, letting $P(x)=c_1p_1(x)+\cdots+c_sp_s(x)$ we have
\begin{equation}\label{differential_polynomial_eqn}
\eta=P(x)\omega_0=\frac{P(x)}{2y}dx.
\end{equation}

In order to express $\eta$ in the form $\eta=w(t)dt$ it suffices to find the Laurent series for $x$ and $y$ in powers of $t$. Indeed, writing $x=x(t)$ and $y=y(t)$ we can find the Laurent series, say $\zeta(t)$, for the rational function $P(x)/2y$. Then \eqref{differential_polynomial_eqn} yields $\eta=\zeta(t)x'(t)dt$, where $x'(t)$ is the formal derivative of $x(t)$.

It remains to explain how the Laurent series for $x$ and $y$ are computed. Suppose first that $Q=(a,b)$ is an affine point with $b\ne 0$. The local parameter at $Q$ is $t=x-a$, so the power series for $x$ is $x(t)=a+t$. Now, Hensel's Lemma implies that there is a unique power series $y(t)\in\Q_p\llbracket t\rrbracket$ such that $y(t)^2=f(t+a)$ and $y(0)=b$; this is the series for $y$.

Similarly, if $Q$ is of the form $Q=(a,0)$, then the local parameter is $t=y$, which gives the power series for $y$. The series for $x$ is the unique power series $x(t)\in\Q_p\llbracket t\rrbracket$ such that $t^2=f(x(t))$ and $x(0)=a$.

Finally, if $Q$ is the point at infinity, then the local parameter is $t=y/x^{g+1}$. Let $h(x)$ be the polynomial $h(x)=x^{2g+2}\cdot f(1/x)$, so that the complementary affine patch of $C$ is given by $v^2=h(u)$. Let $u(s)$ and $v(s)$ be the power series for $u$ and $v$ at the point $(0,0)$. The Laurent series for $x$ and $y$ are then $x(t)=1/u(t)$ and $y(t)=t\cdot x(t)^{g+1}$. 

\smallskip
\noindent\underline{\textit{Step \ref{strassmann_step}.}} In view of Remark \ref{strassmann_bound_rem}, once enough of the terms of the power series $F(t)$ have been determined, it is a straightforward calculation to find the Stra{\ss}mann bound.

We proceed now to apply all of the methods discussed above to determine the rational points on five curves.

\subsection{Rational points on five curves}\label{five_curves_section}
 
\begin{prop}\label{genus4_curve} Let $C$ be the hyperelliptic curve over $\Q$ defined by the equation $y^2=f(x)\cdot g(x)$, where
\begin{align*}
f(x)&=(x^2 + 1)(3x^2 + 2x - 3),\\
g(x)&=43x^6 + 48x^5 - 81x^4 - 80x^3 + 81x^2 + 48x - 43.
\end{align*} 
The set of affine rational points on $C$ is $\{(\pm 1,\pm 8)\}$.
\end{prop}
\begin{proof}
The polynomials $f$ and $g$ have no rational root. Hence, if $(x_0,y_0)$ is an affine rational point on $C$, then $f(x_0)$ and $g(x_0)$ are nonzero, and since their product is a square, they must have the same squarefree part, say $d$. Thus, we have $du_0^2=f(x_0)$ and $dv_0^2=g(x_0)$ for some $u_0,v_0\in\Q$. 

For every squarefree integer $d$, let $C_d\subset\A^3=\Spec\Q[x,u,v]$ be the curve
\[
C_d:\;\begin{cases}
\;du^2=f(x),\\
\;dv^2=g(x).
\end{cases}
\]
Note that there is a map $C_d\to C,\;\; (x,u,v)\mapsto (x, duv)$.
The above argument shows that every affine rational point on $C$ lifts to $C_d(\Q)$ for some $d$.

We can narrow down the possibilities for $d$. Suppose that $(x_0,y_0)\in C(\Q)$ lifts to a point $(x_0,u_0,v_0)\in C_d(\Q)$. If $p$ is a prime dividing $d$, it is easy to see that $x_0, u_0, v_0\in\Z_p$. Reducing the equations $du_0^2=f(x_0)$ and  $dv_0^2=g(x_0)$ modulo $p$, we conclude that $f$ and $g$ have a common root modulo $p$, and therefore $p$ divides the resultant of $f$ and $g$. The prime divisors of $\Res(f,g)$ are 2, 5, and 17, so $d$ can only be divisible by these primes. It follows that
\[d\in S=\{\pm 1,\pm 2,\pm 5,\pm 10,\pm 17,\pm 34, \pm 85, \pm 170\}.\]
Thus, in order to determine $C(\Q)$ it suffices to determine $C_d(\Q)$ for every $d\in S$. Let $X$ and $Y$ be the hyperelliptic curves
\[X:\;y^2=f(x),\;\;\;Y:\; y^2=g(x),\]
and for every $d\in S$ let $X_d, Y_d$ denote the quadratic twists of $X, Y$ by $d$. Note that there are projection maps
\[C_d\to X_d,\;\; (x,u,v)\mapsto (x,u)\;\;\text{and}\;\;C_d\to Y_d,\;\; (x,u,v)\mapsto (x,v).\]
In particular this implies that $C_d(\Q)=\emptyset$ if $X_d(\Q)=\emptyset$ or $Y_d(\Q)=\emptyset$, and in any case the rational points on $C_d$ can be determined if $Y_d(\Q)$ is known. We will now determine $C_d(\Q)$ for every $d\in S$.

Consider first $d=1$. A search for points of small height on $Y$ yields the points $(1,\pm 4)$. A 2-descent shows that the rank of $\Jac(Y)$ is at most 1. The divisor class $[(1,4)-(1,-4)]$ has infinite order, so in fact the rank of $\Jac(Y)$ is 1. A Chabauty argument combined with a Mordell-Weil sieve now shows that $Y(\Q)=\{(1,\pm 4)\}$. Lifting these points to $C_1$ we conclude that
\[C_1(\Q)=\{(1,\pm 2,\pm 4)\}.\]
By a similar argument we obtain $C_{-1}(\Q)=\{(-1,\pm 2,\pm 4)\}$.

We claim that $C_d(\Q)=\emptyset$ for every $d\in S\setminus\{\pm 1\}$. To prove this we will mostly use local arguments to show that either $X_d(\Q)=\emptyset$ or $Y_d(\Q)=\emptyset$.

For $d=\pm 17$ we have $Y_d(\Q_{17})=\emptyset$, and for $d=\pm 85$, $Y_d(\Q_5)=\emptyset$. For $d\in\{\pm 2,\pm 10,\pm 34,\pm 170\}$ we find that $X_d(\Q_2)=\emptyset$. Finally, for $d=\pm 5$, a 2-cover descent calculation for the curve $Y_d$  returns an empty fake 2-Selmer set (in the terminology of \cite{bruin-stoll_2cover_descent}); hence $Y_d(\Q)=\emptyset$. 

To summarize, we have shown that $C_d(\Q)=\emptyset$ for every $d\in S$ except $d=\pm 1$. Thus, every affine rational point on $C$ lifts to a rational point on either $C_1$ or $C_{-1}$. Mapping the points $(1,\pm 2,\pm 4)\in C_1(\Q)$ and $(-1,\pm 2,\pm 4)\in C_{-1}(\Q)$ to $C(\Q)$ we conclude that $C(\Q)=\{(\pm 1,\pm 8)\}$.
\end{proof}

\begin{prop}\label{genus3_chabauty} Let $C$ be the hyperelliptic curve over $\Q$ defined by
\[y^2=f(x):=(x^3+4)(x^4 + 8x^3 + 16x^2 + 64).\]
The set of affine rational points on $C$ is $\{(0,\pm 16), (4,\pm 272)\}$.
\end{prop}
\begin{proof}
A search for points on $C$ yields the five points $\infty, (0,\pm 16), (4,\pm 272)$.
We aim to show that $\# C(\Q)\le 5$, so that these are all the rational points on $C$. To prove this bound we will use a Chabauty-Coleman argument with the prime $p=11$, which is a prime of good reduction for $C$. We apply the method as described in \S\ref{hyperelliptic_chabauty_section}.

Let $J$ be the Jacobian variety of $C$. Our first step is to find generators for a subgroup of finite index in $J(\Q)$. By a 2-descent calculation we find that the rank of $J(\Q)$ is at most 1. The divisor class $D=[\infty-(0,16)]$ can be shown to have infinite order, so the rank is equal to 1, and $D$ generates a subgroup of finite index in $J(\Q)$.

Next we find a suitable basis for the annihilator of $J(\Q)$. Using the basis
\[\omega_0=(1/2y)dx,\; \omega_1=x\omega_0,\;\omega_2=x^2\omega_0\]

for the global differentials on $C$ and computing the entries of the matrix
\[M=\left(\int_{\infty}^{(0,16)}\omega_0,\;\int_{\infty}^{(0,16)}\omega_1,\; \int_{\infty}^{(0,16)}\omega_2\right)=(a,b,c),\]
we obtain
\begin{align*}
a&=8\cdot11^3 + O(11^5),\\
b&=10\cdot11 + 5\cdot11^3 + 5\cdot11^4 + O(11^5),\\
c&=3\cdot11 + 5\cdot11^2 + 6\cdot11^4  + O(11^5).
\end{align*}
The vectors $(b,-a,0)$ and $(0, -c,b)$ form a basis for the kernel of $M$, so Lemma \ref{annihilator_lem} implies that the differentials $b\omega_0-a\omega_1$ and $-c\omega_1+b\omega_2$ form a basis for $\Ann J(\Q)$. Scaling these differentials so that they will be nonzero modulo 11, we obtain a new basis for $\Ann J(\Q)$:
\[\alpha=\frac{1}{11}(b\omega_0-a\omega_1)\;\;\;\text{and}\;\;\;\beta=\frac{1}{11}(-c\omega_1+b\omega_2).\]

Next we use the differentials $\alpha$ and $\beta$ to find upper bounds for the number of rational points within each residue disk in $C(\Q_{11})$.

The points on the reduction of $C$ modulo 11 are the following:
\[C(\F_{11})=\{\infty, (0,\pm 5), (1,\pm 4), (4,\pm 3), (5,\pm 5), (6,0), (8,\pm 2)\}.\]

For each residue disk $\calR$ above one of these points we will choose a differential $\eta\in\Ann J(\Q)$ and compute an upper bound $N$ for the number of zeros of the function $\lambda_{\eta}(Q)=\int_{\infty}^Q\eta$ in $\calR$. Our goal is to choose $\eta$ in such a way that either $N=0$, in which case $C(\Q)\cap\calR=\emptyset$, or $N=1$ when we know that $\calR$ contains a rational point. By trial and error, choosing various linear combinations of $\alpha$ and $\beta$, we have found such a differential $\eta$ for every residue disk. We summarize the results of our calculations in Table \ref{chabauty3_table} below. The table contains, for every differential $\eta\in\{\alpha, \beta, 4\alpha+\beta, 7\alpha+\beta\}$ and every point $\tilde Q\in C(\F_{11})$, the computed upper bound for the number of zeros of the function $\lambda_{\eta}$ in the residue disk above $\tilde Q$.

\medskip
\begin{table}[h!]
\centering
\begin{tabular}{|c|c|c|c|c|c|c|c|} 
 \hline
& $\infty$ & $(0,\pm 5)$ & $(1,\pm 4)$ & $(4,\pm 3)$ & $(5,\pm 5)$ & $(6,0)$ & $(8,\pm 2)$ \\
 \hline
$\alpha$ & 5 & 1 & 1 & 1 & 1 & 1 & 1\\
 \hline
 $\beta$ & 1 & 2 & 1 & 1 & 1 & 1 & 0\\
 \hline
 $4\alpha+\beta$ & 1 & 1 & 1 & 1 & 0 & 1 & 1\\
 \hline
 $7\alpha+\beta$ & 1 & 1 & 0 & 1 & 1 & 1 & 1\\
 \hline
\end{tabular}
\bigskip
\caption{Bounds for rational points in each residue disk.}
\label{chabauty3_table}
\end{table}

For the residue disks $\calR$ containing a known rational point -- namely those above $\infty$, $(0,\pm 5),$ and $(4,\pm 3)$ -- we can conclude from the above data that $\#(\calR\cap C(\Q))=1$. Indeed, either $\alpha$ or $\beta$ provides an upper bound of 1 for the number of rational points in $\calR$.

For the residue disks $\calR$ above $(1,\pm 4)$, $(5,\pm 5)$, and $(8,\pm 2)$ we can conclude that $C(\Q)\cap\calR=\emptyset$, since at least one of the chosen differentials provides an upper bound of 0.

At this stage we know that $5\le\#C(\Q)\le 6$, since the data does not rule out rational points above $(6,0)$. However, $\#C(\Q)$ must be odd since $C$ has no affine rational Weierstrass point. Therefore $\#C(\Q)=5$, which completes the proof.

As an example, we include the details of one of the computations needed to fill in Table \ref{chabauty3_table}. For the reader interested in reproducing our results, the \textsc{Sage} code used for all of the calculations is available in \cite{chabauty_code}.

Let us show that integration of $\eta=7\alpha+\beta$ rules out rational points in the residue disk above the point $\tilde Q=(1,4)\in C(\F_{11})$. Lifting $\tilde Q$ to a point in $C(\Q_{11})$ we obtain the point
\[Q=(1,4 + 9\cdot 11 + 6\cdot 11^2 + 7\cdot 11^3 + 8\cdot 11^4 + O(11^5)).\]

A local parameter at $Q$ is $t=x-1$. Expanding $x$ and $y$ in powers of $t$ we have $x(t)=1+t$ and $y(t)=y_0+y_1t+y_2t^2+y_3t^3+O(t^4)$, where
\begin{align*}
y_0 &= 4 + 9 \cdot 11 + 6 \cdot 11^{2} + 7 \cdot 11^{3} + O(11^{4}),\\
y_1 &= 9 + 6 \cdot 11 + 7 \cdot 11^{3} + O(11^{4}),\\
y_2 &= 3 + 9 \cdot 11^{2} + 11^{3} + O(11^{4}),\\
y_3 &= 9 + 8 \cdot 11 + 2 \cdot 11^{2} + 9 \cdot 11^{3}+ O(11^{4}).\\
\end{align*}

\vspace{-5mm}
Expressing $\eta$ in the form $\eta=\frac{P(x)}{2y}dx$ and using the above expansions for $x$ and $y$, we find that $\eta=w(t)dt=(c_0+c_1t+c_2t^2+c_3t^3+O(t^4))dt$, where
\begin{align*}
c_0 &= 3 \cdot 11 + 2 \cdot 11^{2} + 8 \cdot 11^{3} + O(11^{4}),\\
c_1 &= 9 + 2 \cdot 11 + 3 \cdot 11^{2} + O(11^{4}),\\
c_2 &= 3 + 4 \cdot 11 + 10 \cdot 11^{2} + 2 \cdot 11^{3} + O(11^{4}),\\
c_3 &= 3 + 3 \cdot 11 + 4 \cdot 11^{2} + 7 \cdot 11^{3} + O(11^{4}).\\
\end{align*}

\vspace{-5mm}
Next, we construct the power series $F(t)=\int_{\infty}^Q\eta+I(pt)$, where $I(t)$ is the formal antiderivative of $w(t)$. We obtain $F(t)=b_0+b_1t+b_2t^2+O(t^3)$, where
\begin{align*}
b_0 &= 4 \cdot 11 + 5 \cdot 11^{2} + 3 \cdot 11^{3} + O(11^{5}),\\
b_1 &= 3 \cdot 11^{2} + 2 \cdot 11^{3} + 8 \cdot 11^{4} + O(11^{5}),\\
b_2 &= 10 \cdot 11^{2} + 6 \cdot 11^{3} + 11^{4} + O(11^{5}).\\
\end{align*}

\vspace{-5mm}
The bound \eqref{power_series_precision} is satisfied with $n=2$, so these coefficients are enough  to determine the Stra{\ss}mann bound for $F$. Clearly $b_0$ has the minimal valuation among the coefficients of $F$, and there is no other coefficient with equal valuation. Hence the Stra{\ss}mann bound is 0, as claimed.
\end{proof}

\begin{prop}\label{genus4_chabauty}
Let $C$ be the hyperelliptic curve over $\Q$ defined by
\[y^2 = (x^2 + 1)(x^2 - 2x - 1)(x^6 - 3x^4 - 16x^3 + 3x^2 - 1).\]
The set of affine rational points on $C$ is
\[\{(0, \pm 1), (\pm 1, \pm 8), (3, \pm 40), (-1/3, \pm 40/243)\}.\]
\end{prop}
\begin{proof}
We claim that the number of rational points on $C$ must be a multiple of 4. Let $\sigma$ denote the hyperelliptic involution on $C$. Note that $C$ has an automorphism $\tau$ given by $\tau(x,y)=(-1/x,y/x^5)$. Since $\tau^2=\sigma$, $\tau$ has order 4. Neither $\tau$, $\sigma$, nor $\sigma\tau$ have rational fixed points, so the orbit of every point in $C(\Q)$ under $\tau$ must consist of exactly four points. This implies our claim. 

We will now use the method of Chabauty and Coleman to show that $\#C(\Q)\le 14$ and therefore, in view of the above, $\#C(\Q)\le 12$. Since we already know 12 rational points on $C$ (namely the two points at infinity and the ten affine points listed above), this will prove that $\#C(\Q)= 12$, from which the proposition follows.

For the Chabauty argument we use the prime $p=13$. Given that $C$ is not defined by a model of odd degree, we begin by making a change of variables in order to be able to apply the Coleman integration algorithms from \cite{balakrishnan-bradshaw-kedlaya}. Let $i$ denote a square root of $-1$ in the field $\Q_{13}$. The change of variables

\[(x,y)\mapsto\left(\frac{32-96i}{x+i},\frac{y(32-96i)^4}{(x+i)^5}\right)\]

induces an isomorphism $\phi:C\to X$ (defined over $\Q_{13}$), where $X$ is a hyperelliptic curve over $\Q_{13}$ defined by an equation of the form
\[y^2=x^9 + (144i + 432)x^8+\cdots+2^{44}(336i-527).\]

The map $\phi$ gives rise to a pullback isomorphism
\[\phi^{\ast}: H^0(X,\Omega_{X/\Q_{13}}^1)\to H^0(C,\Omega_{C/\Q_{13}}^1).\]

To simplify notation we will henceforth write $\omega^{\ast}$ instead of $\phi^{\ast}(\omega)$. 

Using the model $X$ and the integration algorithms from \cite{balakrishnan-bradshaw-kedlaya} we can compute Coleman integrals on $C$. Indeed, by the change of variables formula (see \cite[Thm. 2.7]{coleman_integrals}) we have
\[\int_P^Q\omega^{\ast}=\int_{\phi(P)}^{\phi(Q)}\omega\]

for all points $P,Q\in C(\Q_p)$ and every global differential $\omega$ on $X$. With this observation in mind, we proceed to find an upper bound for $\#C(\Q)$.

Let $J$ be the Jacobian variety of $C$. A 2-descent computation shows that the rank of $J(\Q)$ is at most 2. Since the genus of $C$ is 4, this guarantees that $\Ann J(\Q)$ is nontrivial. We will now find a basis for the annihilator.

Note that the two rational points at infinity on $C$ are mapped by $\phi$ to the points $(0,\pm (32-96i)^4)$ on $X$. Let $B=(0,(32-96i)^4)\in X(\Q_{13})$ and $A=\phi^{-1}(B)\in C(\Q)$.

Let $\omega_0,\ldots, \omega_3$ be the differentials on $X$ given by $\omega_i=(x^i/2y)dx$. Computing the entries of the $2\times 4$ matrix
\[
M=\begin{pmatrix}
\int_A^{(0,1)}\omega_0^{\ast} & \int_A^{(0,1)}\omega_1^{\ast} & \int_A^{(0,1)}\omega_2^{\ast} & \int_A^{(0,1)}\omega_3^{\ast}\\
\int_A^{(1,8)}\omega_0^{\ast} & \int_A^{(1,8)}\omega_1^{\ast} & \int_A^{(1,8)}\omega_2^{\ast} & \int_A^{(1,8)}\omega_3^{\ast}\\
\end{pmatrix}
\]

we find that the 13-adic valuations of the entries of $M$ are $2,1,1,1$ in the first row and $1,1,1,3$ in the second row. It follows that the rows of $M$ are linearly independent, and therefore $M$ has rank 2.

The fact that the entries of $M$ are nonzero implies that the divisor classes $D_1=[(0,1)-A]$ and $D_2=[(1,8)-A]$ have infinite order in $J(\Q)$. Moreover, since $M$ has rank 2, $D_1$ and $D_2$ must be linearly independent over $\Z$. (Any relation of the form $mD_1=nD_2$ for nonzero integers $m$ and $n$ would imply that the rows of $M$ are linearly dependent.) We can therefore conclude that $D_1$ and $D_2$ generate a subgroup of finite index in $J(\Q)$.

Since $M$ has rank 2, the kernel of $M$ has dimension 2. Computing a basis for the kernel we obtain the vectors $v=(1,0,a,b)$ and $w=(0,1,c,d)$, where
\begin{align*}
a &= 3 + 6 \cdot 13 + 3 \cdot 13^{2} + 6 \cdot 13^{3} + O(13^{4}),\\
b &= 11 + 11 \cdot 13 + 11 \cdot 13^{2} + 6 \cdot 13^{3} + O(13^{4}),\\
c &= 3 + 13 + 8 \cdot 13^{2} + 8 \cdot 13^{3} + O(13^{4}),\\
d &= 9 \cdot 13 + 5 \cdot 13^{2} + 8 \cdot 13^{3} + O(13^{4}).\\
\end{align*}

\vspace{-5mm}
By Lemma \ref{annihilator_lem}, the differentials 
\[\alpha^{\ast}:=\omega_0^{\ast}+a\omega_2^{\ast}+b\omega_3^{\ast}\;\;\;\text{and}\;\;\;\beta^{\ast}:=\omega_1^{\ast}+c\omega_2^{\ast}+d\omega_3^{\ast}\]

form a basis for the annihilator of $J(\Q)$. Thus we have
\begin{equation}\label{genus4_chabauty_annihilator}
\int_{B}^{\phi(P)}\alpha\;=\;\int_{B}^{\phi(P)}\beta\;=\;0\;\;\text{for all}\;\;P\in C(\Q).
\end{equation}

Let $S$ denote the set $\phi(C(\Q))$. We aim to show that $\#S\le 14$. For every global differential $\eta$ on $X$ we define a function $\lambda_{\eta}:X(\Q_{13})\to \Q_{13}$ by the formula $\lambda_{\eta}(Q)=\int_{B}^Q\eta$. By \eqref{genus4_chabauty_annihilator}, if $\eta$ is a linear combination of $\alpha$ and $\beta$, then $\lambda_{\eta}$ vanishes on $S$.

For every residue disk $\calR\subset X(\Q_{13})$ we will choose a differential $\eta$ on $X$, constructed as a linear combination of $\alpha$ and $\beta$, and determine an upper bound for the number of zeros that the function $\lambda_{\eta}$ has in $\calR$. Since $X$ is hyperelliptic and is given by a model of odd degree, a bound can be computed as discussed in \S\ref{hyperelliptic_chabauty_section}. Let $N_{\calR,\eta}$ denote the bound thus obtained.

We summarize the results of our calculations below. The code used for all of these computations is available in \cite{chabauty_code}.

Mapping the 12 known rational points on $C$ into $X$ using the map $\phi$ and then reducing modulo 13, we obtain the set
\[\tilde S=\{(0,\pm 4),(4,\pm 6),(5,\pm 2),(8,\pm 4),(9,\pm 6),(12,\pm 3)\}.\]

The set of $\F_{13}$-points on the reduction of $X$ is given by

\[X(\F_{13})=\tilde S\cup\{(2,0),(6,\pm 3),(11,\pm 2),\infty\}.\]

For every residue disk $\calR$ above a point in $\tilde S$ we compute $N_{\calR,\alpha}=1$. Hence each one of these 12 residue disks contains exactly one point in $S$.

For the residue disks above $\infty$ and above $(2,0)$ we also obtain $N_{\calR,\alpha}=1$. Hence each one of these disks contains at most one point in $S$.

Finally, for the residue disks above $(6,\pm 3)$ we find that $N_{\calR,\alpha+5\beta}=0$, and for those above $(11,\pm 2)$ we obtain $N_{\calR,\alpha+\beta}=0$. Hence these residue disks do not intersect $S$.

We conclude that $\#S\le 14$ and therefore $\#C(\Q)\le 14$, as desired.
\end{proof}

\begin{prop}\label{genus6_descent}
Let $C$ be the hyperelliptic curve over $\Q$ defined by the equation $y^2=f(x)\cdot g(x)$, where
\begin{align*}
f(x) &= (x^4-1)(x^2 - 2x - 1),\\
g(x) &= x^8 - 8x^5 + 2x^4 + 8x^3 + 1.
\end{align*}
The set of affine rational points on $C$ is $\{(\pm 1,0), (0,\pm 1)\}$.
\end{prop}
\begin{proof} For every squarefree integer $d$, let $C_d\subset\A^3=\Spec\Q[x,u,v]$ be the curve
\[
C_d:\;\begin{cases}
\;du^2=f(x),\\
\;dv^2=g(x).
\end{cases}
\]

Note that there is a map $C_d\to C$ given by $(x,u,v)\mapsto (x, duv)$. Suppose that $(x_0,y_0)$ is an affine rational point on $C$ with $x_0\ne\pm 1$. Then $f(x_0)$ and $g(x_0)$ are nonzero and $y_0^2=f(x_0)g(x_0)$. Thus $f(x_0)$ and $g(x_0)$ have the same squarefree part, say $d$, so there exist $u_0,v_0\in\Q$ such that $(x_0,u_0,v_0)\in C_d(\Q)$. 

The prime divisors of $\Res(f,g)$ are $2,5$, and 17, so $d$ can only be divisible by these primes. Furthermore, $g(x)$ only takes positive values for $x\in\R$, so $d$ must be positive. Therefore
\[d\in S=\{1,2,5,10,17,34,85,170\}.\]

Let $X$ and $Y$ be the hyperelliptic curves $y^2=f(x)$ and $y^2=g(x)$, respectively. We denote by $Y_d$ the quadratic twist of $Y$ by $d$. Since $dv_0^2=g(x_0)$, the curve $Y_d$ has a rational point. Now, for $d\in S\setminus\{1\}$ we find that $Y_d(\Q_2)=\emptyset$ if $d\ne 17$, and $Y_d(\Q_3)=\emptyset$ if $d=17$. Hence we must have $d=1$.

Since $d=1$, then $u_0^2=f(x_0)$, so $x_0$ is the first coordinate of a rational point on $X$. Now, the Jacobian of $X$ has rank 0; this allows us to show that the affine rational points on $X$ are $(\pm 1, 0)$ and $(0,\pm 1)$. Since we are assuming $x_0\ne\pm 1$, this implies that $x_0=0$. Solving $y_0^2=f(0)g(0)$ we obtain $y_0=\pm 1$ and therefore $(x_0,y_0)=(0,\pm 1)$.
\end{proof}

\begin{prop}\label{genus8_descent_chabauty}
Let $C$ be the hyperelliptic curve over $\Q$ defined by the equation $y^2=f(x)\cdot g(x)$, where
\begin{align*}
f(x) &= (x^2+1)(x^8 - 8x^5 + 2x^4 + 8x^3 + 1),\\
g(x) &= (x^2 - 2x - 1)(x^6 - 3x^4 - 16x^3 + 3x^2 - 1).
\end{align*}
The set of affine rational points on $C$ is $\{(\pm1,\pm16), (0,\pm 1)\}$.
\end{prop}
\begin{proof}
Suppose that $(x_0,y_0)$ is a rational point on $C$ with $x_0\ne 0$, and let $d$ be the common squarefree part of $f(x_0)$ and $g(x_0)$. The prime divisors of $\Res(f,g)$ are 2, 5, and 17. Moreover, $f(x)$ only takes positive values for $x\in\R$, so $d$ must be positive. Therefore
\[d\in S=\{1,2,5,10,17,34,85,170\}.\]

Let $X$ and $Y$, respectively, denote the hyperelliptic curves defined by $y^2=f(x)$ and $y^2=g(x)$.

For $d=5, 10, 17$, and $34$ we find that $X_d(\Q_7)=\emptyset$, and for $d=170$ we have $Y_d(\Q_5)=\emptyset$. For $d=85$ the curve $Y_d$ has an empty fake 2-Selmer set, so $Y_d(\Q)=\emptyset$. Thus $d=1$ or 2.

Suppose that $d=1$, so that there exists $z_0\in\Q$ such that $(x_0,z_0)\in Y(\Q)$. Let $H$ be the hyperelliptic curve defined by
\[t^2=s^6 - s^5 + s^4 - 3s^3 + 2s^2 - 2s + 2.\]

There is a map $\phi:Y\dashedrightarrow H$ given by
\[\phi(x,y)=\left(\frac{x^2-1}{2x},\frac{y(x^2+1)}{8x^3}\right).\]

Since $x_0\ne 0$, $\phi$ is defined at the point $(x_0,z_0)$, and $\phi(x_0,z_0)\in H(\Q)$. Now, the Jacobian of $H$ has rank 1, and  a Chabauty computation shows that $H$ has only one affine rational point, namely $(1,0)$. Pulling this point back to $Y$ we obtain no rational point on $Y$, which is a contradiction.

The only remaining possibility is that $d=2$. In this case there exists $z_0\in\Q$ such that $(x_0,z_0)\in Y_2(\Q)$. The same map $\phi$ defined above maps $Y_2$ to the quadratic twist $H_2$. The Jacobian of $H_2$ has rank 1, and a Chabauty computation shows that the only rational points on $H_2$ are $(0,\pm 1)$ and $(1,0)$. Pulling these points back to $Y_2$ we obtain the points $(\pm1, \pm 4)$. We must therefore have $x_0=\pm 1$, and the proposition follows.
\end{proof}

\subsection*{Acknowledgements} The author would like to thank Joseph Wetherell for sharing \textsc{Sage} code used in Chabauty-Coleman calculations, and Imme Arce for her help in preparing the figures.

\appendix
\section{Galois group data}\label{galois_data_appendix}

We define here the groups appearing in our main result, Theorem \ref{phi4_factorization_galois_thm}. Every group is presented as a subgroup of the symmetric group $S_{12}$ by specifying generators for the group. Further information about these groups can be accessed by searching the \textsc{Small Groups} library \cite{small_groups}, which is distributed with the \textsc{Gap} and \textsc{Magma} computer algebra systems. Every group in this library is identified by a pair $(o, n)$ where $o$ is the order of the group; for each group listed below we provide the corresponding identification pair.

\subsection*{\underline{Group $\calW$}} \;\; ID: $(384, 5557)$. Generators:
\begin{align*}
&(1,6,11,2,7,12,3,8,9,4,5,10),\\
&(1,7,4,6,3,5,2,8)(9,10,11,12).
\end{align*}

\subsection*{\underline{Group $\calG$}} \;\; ID: $(192, 944)$. Generators:
\begin{align*}
&(5, 8, 7, 6)(9, 12, 11, 10),\\
&(1, 3)(2, 4)(5, 10, 7, 12)(6, 11, 8, 9), \\
&(1, 12, 7)(2, 9, 8)(3, 10, 5)(4, 11, 6).
\end{align*}

\subsection*{\underline{Group $\calH$}} \;\;ID: $(128, 490)$. Generators:
\begin{align*}
&(1, 10, 4, 9, 3, 12, 2, 11),\\
&(1, 9, 3, 11)(2, 10, 4, 12),\\
&(1, 11, 4, 10, 3, 9, 2, 12)(5, 8, 7, 6).
\end{align*}    

\subsection*{\underline{Group $\calI$}} \;\; ID: $(64, 101)$. Generators:
\begin{align*}
&(1, 4, 3, 2), \\
&(9, 10, 11, 12), \\
&(1, 10, 4, 9, 3, 12, 2, 11), \\
&(1, 10)(2, 11)(3, 12)(4, 9)(5, 7)(6, 8).
\end{align*}

\bibliography{ref_list}
\bibliographystyle{amsplain}
\end{document}